\documentclass[twoside,11pt]{article}
\usepackage{latexsym}
\usepackage{epsfig}
\usepackage{amsfonts}
\usepackage{amsthm,amsmath,amsfonts,amssymb}
\numberwithin{equation}{section}

\newcommand{\opm}{\operatorname{op}}
\newcommand{\Res}{\operatorname{Res}}
\newcommand{\Diff}{\operatorname{Diff}}
\newcommand{\Pf}{\operatorname{Pf.}}
\newcommand{\Aop}{\operatorname{\cal A}}
\newcommand{\Pop}{\operatorname{\cal P}}

\newcommand{\Hop}{\operatorname{\cal H}}

\def\dbar{{\mathchar'26\mkern-12mud}}

\begin{document}
\newtheorem{assumption}{Assumption}
\newtheorem{proposition}{Proposition}
\newtheorem{definition}{Definition}
\newtheorem{lemma}{Lemma}
\newtheorem{theorem}{Theorem}
\newtheorem{observation}{Observation}
\newtheorem{remark}{Remark}
\newtheorem{corollary}{Corollary}
\newtheorem{example}{Example}

\title{Fundamental Solutions and Green's Functions for certain Elliptic Differential Operators from a Pseudo-Differential Algebra}

\author{Heinz-J{\"u}rgen Flad$^\dag$, Gohar Flad-Harutyunyan$^\ast$,
\ \\
$^\dag${\small Institut f\"ur Numerische Simulation, Universit\"at Bonn,
Friedrich-Hirzebruch-Allee 7 53115 Bonn}\\
$^\ast${\small Fachgruppe Mathematik, RWTH Aachen, Pontdriesch 14-16, 52062 Aachen}
}
\maketitle

\begin{abstract}
\noindent
We have studied possible applications of a particular pseudo-differential algebra in singular analysis for the construction of fundamental solutions and Green's functions of a certain class of elliptic partial differential operators. The pseudo-differential algebra considered in the present work, comprises degenerate partial differential operators on stretched cones which can be locally described as Fuchs type differential operators in appropriate polar coordinates. We present a general approach for the explicit construction of their parametrices, which is based on the concept of an asymptotic parametrix, introduced in \cite{FHS16}.
For some selected partial differential operators, we demonstrate the feasibility of our approach by an explicit calculation of fundamental solutions and Green's functions from the corresponding parametrices. In our approach, the Green's functions are given in separable form, which generalizes the Laplace expansion of the Green's function of the Laplace operator in three dimensions. As a concrete application in quantum scattering theory, we construct a fundamental solution of a single-particle Hamilton operator with singular Coulomb potential.       
\end{abstract}

\section{Introduction}
\label{Section-intro}
In the field of partial differential equations, fundamental solution and related Green's functions are a versatile tool with a wide range of applications in mathematics, physics and engineering. Whereas the notion of a fundamental solution is uniquely defined, care has to be taken concerning the notion of a Green's function, which might differ depending on the context of its application. This is particularly true for physics where the label "Green's function" might refer, e.g., to a classical Green's function in potential theory or to a many-particle Green's function in quantum many-particle theory. Despite some differences in the various notions of a Green's function, there is an essential common feature which links them to fundamental solutions. Therefore it does not seem to be appropriate to give a single mathematically rigorous definition of a Green's function. Instead we will adopt it to the specific situation under consideration. 

In the present work, we want to focus on second order linear partial differential operators of the form
\begin{equation}
 \Aop = \sum_{|\alpha|\leq 2} a_{\alpha}(x) \partial^{\alpha}
\label{Acart}
\end{equation}
in an open domain $0 \in \Omega \subseteq \mathbb{R}^n$.
For such type of operators, let us briefly recall the definition of a fundamental solution and Green's function

\begin{definition}
A distribution $u \in {\cal D}'(\Omega)$ is a fundamental solution of the operator (\ref{Acart}) if it satisfies the equation
\[
 \Aop u=\delta
\] 
in a distributional sense, with respect to the Dirac-distribution $\delta$.
\end{definition}

Fundamental solutions have been proven to exist for a large class of differential operators of the form (\ref{Acart}), cf.~\cite{H1,H2,H3} for a detailed account in the case of constant and smooth coefficients.
In the particular case of second order elliptic differential operators with real analytic coefficients in $\Omega := B_R(0)$ a fundamental solution has the form
\begin{equation}
 u(x) = \frac{1}{|x|^{n-2}} f(x) +\ln(|x|) g(x)
\label{uanalytic}
\end{equation}
with $f,g$ real analytic functions in $\Omega$ and $g=0$ for $n$ odd (cf. \cite{FF22,J50,K49}.

As already mentioned before, the more general notion of a Green's function is not well defined and we start with the most general definition.

\begin{definition}
\label{greendef}
A Green's function $G \in {\cal D}'(\Omega)$ of the operator (\ref{Acart}) is a distribution valued function $G: \Omega \ \rightarrow \ {\cal D}'(\Omega)$ which satisfies the distributional equation
\begin{equation}
 \Aop G_{\tilde{x}}=\delta_{\tilde{x}} \ \ \mbox{for all} \ \tilde{x} \in \Omega ,
\label{Gdef}
\end{equation} 
where $\delta_{\tilde{x}}$ denotes the shifted Dirac distribution, i.e., $\delta_{\tilde{x}}(f)=f(\tilde{x})$ for $f \in {\cal D}(\Omega)$.
\end{definition}

In the present work, we want to discuss a novel approach for the construction of fundamental solutions and Green's functions for operators (\ref{Acart}) with possibly singular coefficients $a_{\alpha}$, based on methods from singular analysis,  
To be precise, we allow for a single point $P \in \Omega$, w.l.o.g.~located at the origin, such that $a_{\alpha} \in C^{\infty}(\Omega \setminus P)$.
In order to keep control of the singular behaviour, we restrict ourselves to a specific typ of conical singularity, which can be best characterized by a transformation to an appropriate system of polar coordinates. For notational simplicity, we consider the case $\Omega =\mathbb{R}^n$. Let us introduce the cone ${\cal C}^{\Delta}(X) :=\bar{\mathbb{R}}^{+} \times X /(\{0\} \times X)$ with closed compact, $n-1$-dimensional smooth base manifold $X$ and the corresponding open streched cone ${\cal C}^{\wedge}(X):=\mathbb{R}^{+} \times X$. On ${\cal C}^{\wedge}(X)$, we choose polar coordinates $(r,\phi)$, where $r \in \mathbb{R}^{+}$ and $\phi$ denotes some set of local coordinates on $X$. 
Now consider a homeomorphism $\varphi: {\cal C}^{\Delta} \ \rightarrow \ \mathbb{R}^n$ such that the tip of the cone is mapped on the origin, which induces a diffeomorphism
\begin{equation} 
\varphi|_{{\cal C}^{\wedge}(X)}: {\cal C}^{\wedge}(X) \ \rightarrow \ \mathbb{R}^n \setminus P .
\label{diffeom} 
\end{equation}
From this diffeomorphism, we get a representation of the operator (\ref{Acart}) in polar coordinates of the form 
\begin{equation}
 \tilde{\Aop} = r^{-2} \sum_{j=0}^{2} a_j(r) \biggl(-r \frac{\partial}{\partial r} \biggr)^{j}
\label{Apol}
\end{equation}
where the coefficients $a_j \in \Diff^{2-j}(X)$, $j=0,1,2$, represent partial differential operators, at most of order $2-j$, on $X$.
For our purposes, we have to restrict ourselves to the subset $\Diff_{\deg}^{2}({\cal C}^{\wedge})$ of operators of the form (\ref{Apol}) with coefficients $a_j \in C^{\infty}(\bar{\mathbb{R}}^{+},\Diff^{2-j}(X))$, which belong to the pseudo differential algebra to be discussed below. 
For such partial differential operators (\ref{Apol}), which additionally satisfy the ellipticity conditions of the algebra, parametrices exist. Such a parametrix can be considered as a pseudoinverse of the corresponding partial differential operator modulo so called Green operators, not to be confused with Green's functions, which provide additional asymptotic information. In Section \ref{generalansatz}, we present a general approach for the explicit construction of these parametrices, based on the concept of an asymptotic parametrix, introduced in Refs.~\cite{FHSS11,FHS16}. Within our approach, we want to extract fundamental solutions and Green's functions from the integral kernels of these parametrices. In order to demonstrate its feasibility, we provide in Sections \ref{shiftedLaplace} and \ref{scatteringtheory} some explicit calculations of fundamental solutions and Green's functions from integral kernels of parametrices for some selected partial differential operators. 

An important point which deserves our attention is the fact, that fundamental solutions and Green's functions are by definition distributions in $\Omega \subseteq \mathbb{R}^n$, whereas the kernel of a parametrix and derived quantities are functions on the streched cone ${\cal C}^{\wedge}(X)$ where the origin is excluded by definition. Therefore we will make use of Hadamard's notion of a pseudofunction, cf.~\cite{Schwartz}.
By the diffeomorphism  $x=\varphi(r,\phi)$, cf.~(\ref{diffeom}), a  function $u(r,\phi)$ on ${\cal C}^{\wedge}(X)$ corresponds to a function $   \tilde{u}(x)$ on $\mathbb{R}^n \setminus P$.
A function $u$ on ${\cal C}^{\wedge}(X)$ can be regarded as a regular distribution on $\Omega$
if $\tilde{u}$ can be identified with an element in $L^1_{loc}(\mathbb{R}^n)$, the set of equivalence classes of locally integrable functions in $\mathbb{R}^n$, and therefore with a regular distribution in ${\cal D}'(\mathbb{R}^n)$.
We call this regular distribution the pseudofunction corresponding to $u$ and denote it by $\Pf u$.

Let us briefly illustrate these concepts for the most prominent case in applications which is the Laplace operator $\Delta_3$ in $\mathbb{R}^3$.

\begin{example}
A fundamental solution of the Laplace operator $\Delta_3$ in $\mathbb{R}^3$ is given by $u(x)= -\frac{1}{4\pi|x|}$ and the corresponding Green's function by $G(x,\tilde{x})=u(x-\tilde{x})=-\frac{1}{4\pi|x-\tilde{x}|}$. Given any $f \in {\cal D}(\mathbb{R}^3)$ the convolution
\begin{equation}
 g(x):= \int_{\mathbb{R}^3} G(x,\tilde{x}) f(\tilde{x}) \; d\tilde{x}
\label{gGf}
\end{equation}
satisfies the equation $\Aop g=f$. 

In spherical polar coordinates, the Laplace operator is given by
\[
 \tilde{\Delta}_{3} := \frac{1}{r^2} \biggl[ \biggl(-r \frac{\partial}{\partial r} \biggr)^2 - \biggl(-r \frac{\partial}{\partial r} \biggr) + \Delta_{S^{2}} \biggr]
\]
and its fundamental solution can be expressed as
\[
 u(x)=\Pf -\frac{1}{4\pi r} .
\]
The Green's function $G(\cdot,\tilde{x})$ of the Laplace operator, also known as Newton or Coulomb potential in the physics literature, can be represented in a separable form by the Laplace expansion, i.e.,
\[
 G( \cdot, \tilde{x}) := \left\{ \begin{array}{cc}  \Pf G(\cdot|\tilde{r},\tilde{\phi}) & \mbox{for} \ \tilde{x} = \varphi(\tilde{r},\tilde{\phi}) \neq 0 \\
 \lim_{\tilde{r} \rightarrow 0} \Pf G(\cdot|\tilde{r},\tilde{\phi}) & \mbox{for} \ \tilde{x} =0 \end{array} \right.
\]
with
\begin{equation}
 G(r,\phi|\tilde{r},\tilde{\phi}) = -\sum_{\ell=0}^{\infty} \frac{r_{<}^{\ell}}{r_{>}^{\ell+1}}
 \frac{1}{2\ell+1}
 \sum_{m=-\ell}^{\ell} (-1)^{m} Y_{\ell,m}(\theta,\varphi) Y_{\ell,-m}(\tilde{\theta},\tilde{\varphi})
\label{Coulexp}
\end{equation}
and $r_{<} := \min \{|x|,|\tilde{x}|\}$ and $r_{>} := \max \{|x|,|\tilde{x}|\}$, respectively, where (\ref{Coulexp}) corresponds to a function on ${\cal C}^{\wedge}(S^2) \times {\cal C}^{\wedge}(S^2)$.

In order to achieve a separation of variables, the Laplace expansion is done with respect to the eigenfunctions of the Laplace-Beltrami operator on the sphere $S^2$, so called spherical harmonics, $Y_{\ell,m}$, $\ell=0,1,\ldots$ and $m=-\ell,\ldots,\ell$. By taking the limit $\tilde{r} \rightarrow 0$ in (\ref{Coulexp}) one can easily recover the fundamental solution $u$.
The Laplace expansion greatly facilitates the computation of the convolution (\ref{gGf}) and is therefore of great practical significance in computational physics and chemistry. 
\end{example}

Our change of perspective from cartesian to polar coordinates is a key component of the present work and is motivated by the fundamental solutions (\ref{uanalytic}) for differential operators of the form (\ref{Acart}) with real analytic coefficients and the Laplace expansion of the Green's function (\ref{Coulexp}), where $\mathbb{R}^3$ has been replaced by the cone ${\cal C}^{\Delta}:=\bar{\mathbb{R}}^{+} \times S^2 /(\{0\} \times S^2)$ with base $S^2$ on which the Laplace-Beltrami operator has a pure point spectrum and the eigenfunctions form a complete basis in $L_2(S^2)$.

\subsection{Outline of our approach for the Laplace operator}
\label{Sec-Laplace}
Before we enter into a general discussion for second order differential operators in the cone algebra, we want to exemplify our approach for the Laplace operator in $\mathbb{R}^n$, with $n \geq 3$\footnote{The case $n=2$ is different, according to our discussion in Section \ref{sectionconstruction}}. 
Following our discussion in Section \ref{Section-intro}, there are two different types of representation for the Laplace operator in $n$ dimensions, depending on whether one considers it in $\mathbb{R}^n$ with respect to Cartesian coordinates
 \begin{equation}
 \Delta_{n} := \sum_{j=1}^n \partial_{j}^{2} 
\label{Lcart}
\end{equation}
or some kind of spherical polar coordinates defined on the stretched cone ${\cal C}^{\wedge} := \mathbb{R}^{+} \times S^{n-1}$ with base $S^{n-1}$. 
In spherical polar ($n=3$), and hyperspherical polar ($n\geq 3$) coordinates, the Laplace operator is represented by
\begin{equation}
 \tilde{\Delta}_{n} := \frac{1}{r^2} \biggl[ \biggl(-r \frac{\partial}{\partial r} \biggr)^2 - (n-2)\biggl(-r \frac{\partial}{\partial r} \biggr) + \Delta_{S^{n-1}} \biggr] .
\label{Lspher}
\end{equation}
where $\Delta_{S^{n-1}}$ denotes the Laplace-Beltrami operator on the $n-1$ sphere $S^{n-1}$. The Laplace-Beltrami operator $\Delta_{S^{n-1}}$ has a pure point spectrum with eigenvalues $\lambda_{\ell} =-\ell(\ell+n-2)$, $\ell=0,1,2,\ldots$. For $n=3$ each eigenvalue $\lambda_{\ell}$ has multiplicity $2\ell+1$ and for $n>3$ its multiplicity is given by
\[
 \frac{(\ell +n/2-1)\prod_{j=1}^{n-3}(\ell+j)}{(n/2-1) \cdot (n-3)!} ,
\]
cf.~\cite{STV18} for further details.
In the following, we denote by $P_{\ell}$, $\ell=0,1,2,\ldots$, the projection operators from $L^2(S^{n-1})$ on the corresponding eigenspaces of the eigenvalues $-\ell(\ell+n-2)$. Herewith, we can form the spectral resolution 
\begin{equation}
 \Delta_{S^{n-1}} = - \sum_{\ell=0}^{\infty} \ell(\ell+n-2) P_{\ell}
\label{specres}
\end{equation} 
which is crucial for the following considerations.

The basic idea of our approach is to consider (\ref{Lspher}) as an element of an operator algebra which enables the construction of a parametrix. This can be achieved within the pseudo-differential cone algebra, developed by Schulze and collaborators, cf.~the monographs \cite{ES97,HS08,Schulze98}. Existence of a parametrix for a partial differential operator is intimately connected with its ellipticity.
In the framework of the pseudo-differential calculus, considered in the present work, the notion of ellipticity involves a whole hierarchy of symbols associated to a partial differential operator, cf.~Chapter 10 of \cite{HS08} for a detailed discussion. For unbounded domains, e.g.~${\cal C}^{\wedge}$, one has to take into account the exit behaviour to infinity, as a result, the Laplace operator $\tilde{\Delta}_{n}$ is not elliptic in ${\cal C}^{\wedge}$, because it fails to satisfy the elliptic exit condition.
It is only the shifted Laplacian $\tilde{\Delta}_n-\kappa^2$, discussed in Section \ref{shiftedLaplace}, which satisfies all ellipticity conditions.
For the formal construction of a parametrix outlined below, the exit condition is not essential and does not affect the final result.
A rigorous justification for disregarding it will be given in Section \ref{scatteringtheory}, where we consider a fundamental solution and Green's function for shifted Laplacian and recover the results given below by taking the limit $\kappa \rightarrow 0$. 

In order to construct a parametrix, we have to represent (\ref{Lspher}) as Mellin type pseudo-differential operator, i.e.,
\[
 \tilde{\Delta}_{n} u = r^{-2} \opm_M^{\gamma-\frac{n-1}{2}}(h) u
\]
for $u \in \tilde{{\cal D}}({\cal C}^{\wedge})$, here $\tilde{{\cal D}}({\cal C}^{\wedge}) := \{ \varphi^{\ast} g \ | \ g \in {\cal D}(\mathbb{R}^n)\}$, where $\varphi^{\ast} g$ denotes the pullback under the diffeomorphism (\ref{diffeom}), with operator valued Mellin symbol
\begin{eqnarray}
\nonumber
 h(w) & = & w^2 -(n-2)w + \Delta_{S^{n-1}} \\ \label{mL}
 & = & w^2 -(n-2)w - \sum_{\ell=0}^{\infty} \ell(\ell+n-2) P_{\ell}
\end{eqnarray}
we refer, e.g., to \cite[ Chapter 8]{ES97} for further details. 
For the parametrix we take the ansatz
\[
 \Pop u = r^{2} \opm_M^{\gamma-\frac{n+3}{2}} \bigl( h^{(-1)}(w) \bigr) u
\]
and consider the operator product
\begin{eqnarray*}
 \Pop_n \tilde{\Delta}_{n} & = & r^{2} \opm_M^{\gamma-\frac{n+3}{2}}(h^{(-1)}(w))  r^{-2} \opm_M^{\gamma-\frac{n-1}{2}}(h(w)) \\
 & = & \opm_M^{\gamma-\frac{n-1}{2}}(h^{(-1)}(w+2)) \opm_M^{\gamma-\frac{n-1}{2}}(h(w)) \\
 & = & \opm_M^{\gamma-\frac{n-1}{2}}( h^{(-1)}(w+2)h(w))
\end{eqnarray*}
The operator valued symbol of the parametrix has to satisfy the equation
\[
 h^{(-1)}(w+2)h(w)=1
\]
which can be solved for
\begin{eqnarray}
\nonumber
 h^{(-1)}(w) & = & \frac{1}{h(w-2)} \\ \nonumber
 & = & \frac{1}{(w-2)^2 -(n-2)(w-2) + \Delta_{S^{n-1}}} \\ \label{poles}
 & = & \sum_{\ell=0}^{\infty} \underbrace{\frac{P_{\ell}}{(w-2+\ell)(w-\ell-n)}}_{=: h^{(-1)}_{\ell}(w)} .
\end{eqnarray}
It can be easily that the terms in the sum (\ref{poles}) have only simple poles for $n \geq 3$ but has a pole of order 2 for $n=2$ at $\ell =0$.
It turns out, that for this particular reason, the case $n=2$ is special, therefore we only consider the case $n \geq 3$.

Let $u \in \tilde{{\cal D}}({\cal C}^{\wedge})$, the action of the parametrix $P$ is given by the double integral
\begin{equation}
 r^{2} \opm_M^{\gamma-\frac{n+3}{2}}(h^{(-1)}) u =
 r^2 \int_{\mathbb{R}} \int_{0}^{\infty} \left( \frac{r}{\tilde{r}} \right)^{-(\frac{n+4}{2}-\gamma+i\rho)} h^{(-1)}(\tfrac{n+4}{2}-\gamma+i\rho)\, u(\tilde{r},\phi) \, \frac{d\tilde{r}}{\tilde{r}} \dbar\rho .
\label{Pintegral}
\end{equation}
Here we choose $2-\frac{n}{2} < \gamma < \frac{n}{2}$ which is related to the ellipticity condition of the conormal symbol in the cone algebra. 
Splitting the radial integral into two parts, i.e., $\tilde{r} < r$ and $\tilde{r} > r$, one can apply Cauchy's residue theorem to the spectral resolution (\ref{specres}) of the operator valued symbol.
Taking into account, that all poles in (\ref{specres}) are of order 1,
one can calculate the integral along the complex line by closing the path on the right and left hand side, respectively. 
After some algebraic manipulations one obtains
\begin{eqnarray*}
 \Pop_nu(r,\phi) & = & -\sum_{\ell=0}^{\infty} \int_{S^{n-1}} \int_{0}^{r} \frac{\tilde{r}^{\ell}}{r^{\ell+n-2}}
 \frac{p_{\ell}(\phi|\tilde{\phi})}{2\ell+n-2} u(\tilde{r},\tilde{\phi}) \, \tilde{r}^{n-1} d\tilde{r} \mu_{n-1}(\tilde{\phi}) d\tilde{\phi} \\
 & & -\sum_{\ell=0}^{\infty} \int_{S^{n-1}} \int_{r}^{\infty} \frac{r^{\ell}}{\tilde{r}^{\ell+n-2}}
 \frac{p_{\ell}(\phi|\tilde{\phi})}{2\ell+n-2} u(\tilde{r},\tilde{\phi}) \, \tilde{r}^{n-1} d\tilde{r} \mu_{n-1}(\tilde{\phi}) d\tilde{\phi} .
\end{eqnarray*}
Therefore, the parametrix can be represented by an integral operator
\[
 \Pop_nu(r,\phi) = \int_{S^{n-1}} \int_{0}^{\infty} K_n(r,\phi|\tilde{r},\tilde{\phi}) u(\tilde{r},\tilde{\phi}) \tilde{r}^{n-1} d\tilde{r} \mu_{n-1}(\tilde{\phi}) d\tilde{\phi}
\]
with kernel function
\[
 K_n(r,\phi|\tilde{r},\tilde{\phi}) = -\sum_{\ell=0}^{\infty} \frac{r_{<}^{\ell}}{r_{>}^{\ell+n-2}}
 \frac{p_{\ell}(\phi|\tilde{\phi})}{2\ell+n-2} ,
\]
where $p_{\ell}(\phi|\tilde{\phi})$ denotes the integral kernel of the projection operator $P_{\ell}$, $\ell=0,1,2,\ldots$ and $r_{<} := \min \{r,\tilde{r}\}$, $r_{>} := \max \{r,\tilde{r}\}$, respectively. For $n=3$, using spherical coordinates $\phi=(\theta,\varphi)$, we have 
\[
 p_{\ell}(\theta,\varphi|\tilde{\theta},\tilde{\varphi}) 
 = \sum_{m=-\ell}^{\ell} (-1)^{m} Y_{\ell,m}(\theta,\varphi) Y_{\ell,-m}(\tilde{\theta}, \tilde{\varphi})
\]
and recover the well known Laplace expansion of the Green's function of $\Delta_3$, given by (\ref{Coulexp}).

The Laplace operator is a particularly simple case because the symbol of the parametrix can be derived by direct inversion (\ref{poles}). This was only possible, because all terms which contain derivatives with respect to coordinates of the base $S^{n-1}$ are subsummed in the Laplace-Beltrami operator $\Delta_{S^{n-1}}$ which has especially nice spectral properties that enable the use of a spectral resolution in (\ref{poles}). Furthermore, the direct inversion (\ref{poles}) requires that all coefficients $a_j$, $j=0,1,2$, in the representation (\ref{Apol}), are constants, which is obviously the case for the Laplace operator. However,
this will not be the case in general and we have to rely on an asymptotic parametrix construction \cite{FHS16}, which is considerably more envolved.  

\section{A generalized approach for elliptic partial differential operators in the pseudo-differential algebra}
\label{generalansatz}
The Laplace operator discussed in Section \ref{Sec-Laplace} motivates a generalization of our approach to elliptic second order differential operators on open stretched cones 
${\cal C}^{\wedge}(X)$, with closed compact, $n-1$-dimensional smooth base manifold $X$, of the following form 
\begin{equation}
 \tilde{\Aop} = r^{-2} \left[ \sum_{j=0}^{2} a_j(r) \biggl(-r \frac{\partial}{\partial r} \biggr)^{j} + b(r) \Lambda_X \right]
\label{DefA2}
\end{equation}
and real analytic coefficients $a_j, b \in C^{\infty}(\bar{\mathbb{R}}^{+})$\footnote{In our examples, we consider only the case where these coefficients are polynomials in $r$ of finite order.}. It should be mentioned, that despite the analyticity of the coefficients, (\ref{DefA2}) might be singular, due to prefactor $r^{-2}$ in front of the differential operator.
Furthermore, we assume that $\Lambda_X$ is an elliptic, essentially self-adjoint, second order differential operator, semibounded from below, on the base $X$. Furthermore, we assume that $\Lambda_X$ has a pure point spectrum $\lambda_0,\lambda_1,\lambda_2\ldots$, with lowest eigenvalue $\lambda_0=0$ and a corresponding constant eigenfunction, such that the corresponding eigenfunctions form a complete basis in $L_2(X)$. 

For differential operators (\ref{DefA2}), we want to derive an explicit expression of the kernel of the parametrix in terms of a generalized Laplace expansion of the form
\begin{equation}
 K(r,\phi|\tilde{r},\tilde{\phi}) = \sum_{\ell=0}^{\infty} k_{\ell}(r,\tilde{r},\lambda_{\ell}) p_{\ell}(\phi,\tilde{\phi}) ,
\label{GLaplace}
\end{equation}
where $p_{\ell}(\phi|\tilde{\phi})$ denotes the integral kernel of the projection operator $P_{\ell}$, $\ell=0,1,2,\ldots$, onto the $L_2(X)$ subspace spanned by the eigenfunctions of $\Lambda_X$ for the eigenvalue $\lambda_{\ell}$, i.e.,
\[
 p_{\ell}(\phi,\tilde{\phi}) 
 = \sum_{m=1}^{m_{\ell}} u_{\ell,m}(\phi) \bar{u}_{\ell,m}(\tilde{\phi}),
 \quad \mbox{with} \ \Lambda_X u_{\ell,m} = \lambda_{\ell} u_{\ell,m}, \ \mbox{for} \ m=1,\ldots,m_{\ell} .
\]
In order to deal with $r$ dependent coefficients $a_j,b$, $j=0,1,2$, it becomes necessary to perform an asymptotic parametrix construction, introduced in 
Ref.~\cite{FHS16} from which we take notations and definitions in the following. The differential operator (\ref{DefA2}) can be represented in the cone algebra by the Mellin type pseudo-differential operator 
\[
 \tilde{\Aop} = r^{-2} \opm_M^{\gamma-\frac{n-1}{2}}\bigl(h(r,w,\Lambda_X \bigr) 
\]
where
\[
h(r,w,\Lambda_X) := \sum_{j=0}^{2} a_j(r) w^j+b(r)\Lambda_X
\]
belongs to $C^{\infty}(\bar{\mathbb{R}}^{+},M^{2}_{\cal O}(X))$,
cf.~\cite{Schulze98} for further details.

\subsection{Construction of an asymptotic parametrix}
\label{sectionconstruction}
Given power series representations for the coefficients
\[
 a_j(r) = \sum_{p=0}^{\infty} a_j^{(p)} r^p, \ j=0,1,2 \ \mbox{and} \ b(r) = \sum_{p=0}^{\infty} b^{(p)} r^p
\]
we decompse 
\[
 h(r,w,\Lambda_X) =\sum_{i=0}^{\infty} r^i h_i(w,\Lambda_X)
\]
where 
\[
 h_i(w,\Lambda_X) = \sum_{j=0}^{2} a_j^{(i)} w^j +b^{(i)} \Lambda_X  
\]
denotes the asymptotic symbols of the differential operator (\ref{DefA2}). 
In the following we take $\gamma \in \mathbb{R}$ such that the ellipticity condition with respect to the conormal symbol, which in our case equals $h_0$, is satisfied, i.e., $h_0$ defines isomorphisms $h_0(w,\Lambda_X): \ H^{s}(X) \rightarrow H^{s-2}(X)$ for any $s \in \mathbb{R}$ and $\Re w = \frac{n}{2}-\gamma$. 
With respect to the coefficients, this condition is equivalent to
\[
 a^{(0)}_2 w^2 +a^{(0)}_1 w +a^{(0)}_0 + b^{0} \lambda_{\ell} \neq 0, \quad \ell=0,1,2,\ldots 
\] 
for all $\Re w =\frac{n}{2}-\gamma$. 
For its corresponding asymptotic parametrix, we take
\[
 r^2 \opm_M^{\gamma-\frac{n+3}{2}} \bigl( h^{(-1)}(r,w,\Lambda_X) \bigr)
\]
with formal power series ansatz
\[
 h^{(-1)}(r,w,\Lambda_X) := \sum_{j=0}^{\infty} r^j h^{(-1)}_j(w,\Lambda_X) .
\]
for its symbol. 

\begin{remark}
For notational simplicity, we occasionally represent symbols of pseudo-differential operators by formal power series. Notwithstanding the fact that in the standard calculus such expressions are avoided due to its asymptotic character and lack of convergence in general. In our applications considered in the present work, however, we always perform calculations of parametrix symbols to infinite order and prove convergence of the derived fundamental solutions or Green's functions in hindsight.
\end{remark}

For notational simplicity we suppress the $\Lambda_X$ dependence of the symbols for a while.
As a first step, we commute the $r^{-2}$ factor to the right
\begin{eqnarray}
\nonumber
 P \tilde{\Aop} &=& r^2 \opm_M^{\gamma-\frac{n+3}{2}} \bigl( h^{(-1)}(r,w) \bigr) r^{-2} \opm_M^{\gamma-\frac{n-1}{2}} \bigl( h(r,w) \bigr) \\ \label{PA}
 &=& \opm_M^{\gamma-\frac{n-1}{2}} \bigl( h^{(-1)}(r,w+2) \bigr) \opm_M^{\gamma-\frac{n-1}{2}} \bigl( h(r,w) \bigr) .
\end{eqnarray}
Inserting the power series of the symbols and shifting all powers of $r$ to the right, we obtain
\begin{eqnarray*}
 P \tilde{A} &=& \left[ \sum_{j=0}^{\infty} r^j \opm_M^{\gamma-\frac{n-1}{2}} \bigl( h^{(-1)}_j(w+2) \bigr) \right] \left[ \sum_{i=0}^N r^i \opm_M^{\gamma-\frac{n-1}{2}} \bigl( h_i(w) \bigr) \right] \\
 &=& \sum_{j,i=0}^{\infty} r^{j+i} \opm_M^{\gamma-\frac{n-1}{2}-i} \bigl( h^{(-1)}_j(w+2-i) \bigr) \opm_M^{\gamma-\frac{n-1}{2}} \bigl( h_i(w) \bigr) \\
 &=& \sum_{j,i=0}^{\infty} r^{j+i} \opm_M^{\gamma-\frac{n-1}{2}} \bigl( h^{(-1)}_j(w+2-i) \bigr) \opm_M^{\gamma-\frac{n-1}{2}} \bigl( h_i(w) \bigr) \mod G
\end{eqnarray*}
where the last step is modulo Green's operators, which will be discused in detail below. After the manipulations, we can now define the sympbols of the parametrix in a recursive manner.
For the symbol $h^{(-1)}_0$ we get the equation
\[
 1= \opm_M^{\gamma-\frac{n-1}{2}} \bigl(  h^{(-1)}_0(w+2) \bigr) \opm_M^{\gamma-\frac{n-1}{2}} \bigl( h_0(w) \bigr)
 = \opm_M^{\gamma-\frac{n-1}{2}} \bigl(  h^{(-1)}_0(w+2) h_0(w) \bigr)
\]
and therefore
\begin{equation}
 \Big( h^{(-1)}_0(w+2)\Big) h_0(w) =1 \ \longrightarrow \ h^{(-1)}_0(w) =  h_0^{-1}(w-2) .
\label{h0m1}
\end{equation}
The equations of the symbols $h^{(-1)}_j$, for $j \geq 1$, are
\[
 0= \opm_M^{\gamma-\frac{n-1}{2}} \bigl(  h^{(-1)}_j(w+2) \bigr) \opm_M^{\gamma-\frac{n-1}{2}} \bigl( h_0(w) \bigr) + \sum_{i=1}^{j}
 \opm_M^{\gamma-\frac{n-1}{2}} \bigl( h^{(-1)}_{j-i}(w+2-i) \bigr) \opm_M^{\gamma-\frac{n-1}{2}} \bigl( h_i(w) \bigr) ,
\]
which gives
\[
 0= \bigl( h^{(-1)}_j(w+2)\bigr) h_0(w) + \sum_{i=1}^{j} \bigl(  h^{(-1)}_{j-i}(w+2-i) \bigr) h_i(w) ,
\]
and
\begin{eqnarray}
\nonumber
 h^{(-1)}_j(w) & = & -\left( \sum_{i=1}^{j} \bigl(  h^{(-1)}_{j-i}(w-i) \bigr) \bigl(  h_i(w-2) \bigr) \right) \bigl(  h^{-1}_0(w-2) \bigr) \\ \label{hm1p}
  & = & -\left( \sum_{i=1}^{j} \bigl(  h^{(-1)}_{j-i}(w-i) \bigr) \bigl(  h_i(w-2) \bigr) \right) h^{(-1)}_0(w)
\end{eqnarray}
respectively. In order to derive integral kernels, we have to know the poles of the meromorphic operator valued symbols $h^{(-1)}_j(w)$. 
It can be seen from (\ref{hm1p}) that the poles are recursively defined
such that $h^{(-1)}_j(w)$ has meromorphic terms of the form
\begin{eqnarray}
\nonumber
 j=0: & h^{(-1)}_0(w) \\ \nonumber
 j=1: & \bigl(  h^{(-1)}_0(w-1) \bigr) h^{(-1)}_0(w) \\ \label{hm1poles}
 j=2: & \bigl(  h^{(-1)}_0(w-2) \bigr) \bigl(  h^{(-1)}_0(w-1) \bigr) h^{(-1)}_0(w), \ \bigl(  h^{(-1)}_0(w-2) \bigr) h^{(-1)}_0(w) \\ \nonumber
  j=3: & \bigl(  h^{(-1)}_0(w-3) \bigr) \bigl(  h^{(-1)}_0(w-2) \bigr) \bigl(  h^{(-1)}_0(w-1) \bigr) h^{(-1)}_0(w), \ \bigl( h^{(-1)}_0(w-3) \bigr) \bigl(  h^{(-1)}_0(w-1) \bigr) h^{(-1)}_0(w)  , \\ \nonumber & \bigl(  h^{(-1)}_0(w-3) \bigr) \bigl(  h^{(-1)}_0(w-2) \bigr) h^{(-1)}_0(w), \ \bigl(  h^{(-1)}_0(w-3) \bigr) h^{(-1)}_0(w) \bigr) \\ \nonumber
  \vdots \quad & 
\end{eqnarray}
After re-emerging of $\Lambda_X$ in the symbols of the parametrix, we can perform spectral resolutions
\[
 h^{(-1)}_0(w,\Lambda_X) =  \sum_{\ell=0}^{\infty} h^{(-1)}_0(w,\lambda_{\ell}) P_{\ell}, \quad h_i(w,\Lambda_X) =  \sum_{\ell=0}^{\infty} h_i(w,\lambda_{\ell}) P_{\ell}
\]
\[
 h^{(-1)}_0(w,\lambda_{\ell}) = \frac{1}{\sum_{j=0}^{2} a_j^{(0)} (w-2)^j +b^{(0)} \lambda_{\ell}}, \quad
 h_i(w,\lambda_{\ell}) := \sum_{j=0}^{2} a_j^{(i)} w^i +b^{(i)} \lambda_{\ell}
\]
and use the orthogonality of the projection operators, i.e.~$P_{\ell} P_{\tilde{\ell}} = \delta_{\ell,\tilde{\ell}} P_{\ell}$, to get the coorresponding spectral resolution of the Mellin symbols of the  parametrix
\[
 h^{(-1)}_j(w,\Lambda_X) = \sum_{\ell=0}^{\infty} h^{(-1)}_j(w,\lambda_{\ell}) P_{\ell}
\]
with
\[
  h^{(-1)}_j(w,\lambda_{\ell}) = -\left( \sum_{i=1}^{j} \bigl(  h^{(-1)}_{j-i}(w-i,\lambda_{\ell}) \bigr) \bigl(  h_i(w-2,\lambda_{\ell}) \bigr) \right) h^{(-1)}_0(w,\lambda_{\ell})
\]
where the latter symbols are defined in a recursive manner. According to our previous discussion, the poles of $h^{(-1)}_j(w,\lambda_{\ell})$ are given by the poles of the shifted symbols $h^{(-1)}_0(w-m)$, for $m=0,1,\ldots,j$. The symbol $h^{(-1)}_0(w)$ has poles at 
\[
 w(\lambda_{\ell}) = 2-\frac{a^{(0)}_1}{2a^{(0)}_2} \pm \sqrt{\left(\frac{a^{(0)}_1}{2a^{(0)}_2} \right)^2 -\frac{a^{(0)}_0}{a^{(0)}_2} -\frac{b^{(0)}}{a^{(0)}_2} \lambda_{\ell}} .
\]
In order to simplify our discussion, we want to restrict in the following to two different types of differential operators.

\begin{definition}
Let $\tilde{A}$ be an elliptic differential operator in $\Diff_{\deg}^{2}({\cal C}^{\wedge}(X))$ of the form (\ref{DefA2}), where we assume w.l.o.g. $a^{(0)}_2 =1$ for notational simplicity.
A differential operator is referred as type-A if $\bigl( a^{(0)}_1 \bigr)^2 -4a^{(0)}_0 >0$, $b^{(0)}<0$, and as type-B if
$\bigl( a^{(0)}_1 \bigr)^2 -4a^{(0)}_0 =0$, $b^{(0)}<0$.
\end{definition}

\begin{remark}
According to our definition, the symbol $h_0^{(-1)}$ has two poles of order 1 if it corresponds to a type-A differential operator and a single pole of order 2 if it corresponds to a type-B differential operator.
\end{remark}

\begin{example}
The Laplacian $\Delta_n$ in $\mathbb{R}^n$ is of type-A for $n \geq 3$ and of type-B for $n=2$.
\end{example}

In the present work, we want to focus an type-A differential operators and leave type-B differential operators for a subsequent publication.

\subsubsection{Parametrices for type-A differential operators}
Let us consider asymptotic parametrices for type-A differential operators. For notational simplicity, we introduce the $\lambda_{\ell}$ dependent parameter
\[
 \Delta w(\lambda_{\ell}) := \sqrt{\left( \frac{a^{(0)}_1}{2} \right)^2 -a^{(0)}_0 -b^{(0)} \lambda_{\ell}} .
\]
Let us first consider the operator (\ref{DefA2}) with constant coefficients, which corresponds to the $0$-th order term in the expression of $\tilde{{\cal A}}$ with respect to powers of $r$.
\ \\
\ \\
Performing a similar calculation as in Section 1.1 we get: 
\begin{eqnarray*}
 \Pop_0u(r,\phi) & = & -\sum_{\ell=0}^{\infty} \int_{X} \int_{0}^{r} \frac{\tilde{r}^{w_{1,\ell}-n}}{r^{w_{1,\ell}-2}}
 \frac{p_{\ell}(\phi|\tilde{\phi})}{2\Delta w(\lambda_{\ell})} u(\tilde{r},\tilde{\phi}) \, \tilde{r}^{n-1} d\tilde{r} \mu_{n-1}(\tilde{\phi}) d\tilde{\phi} \\
 & & -\sum_{\ell=0}^{\infty} \int_{X} \int_{r}^{\infty} \frac{\tilde{r}^{w_{2,\ell}-n}}{r^{w_{2,\ell}-2}}
 \frac{p_{\ell}(\phi|\tilde{\phi})}{2\Delta w(\lambda_{\ell})} u(\tilde{r},\tilde{\phi}) \, \tilde{r}^{n-1} d\tilde{r} \mu_{n-1}(\tilde{\phi}) d\tilde{\phi} \\
 & = & -r^{a^{(0)}_1/2} \sum_{\ell=0}^{\infty} \int_{X}  \int_{0}^{\infty} \tilde{r}^{2-n-a^{(0)}_1/2}
 \left( \frac{r_{<}}{r_{>}} \right)^{\Delta w(\lambda_{\ell})}
 \frac{p_{\ell}(\phi|\tilde{\phi})}{2\Delta w(\lambda_{\ell})} u(\tilde{r},\tilde{\phi}) \, \tilde{r}^{n-1} d\tilde{r} \mu_{n-1}(\tilde{\phi}) d\tilde{\phi} ,
\end{eqnarray*}
with 
\[
 w_{1,\ell} := 2-\frac{a^{(0)}_1}{2} + \Delta w(\lambda_{\ell}), \quad
 w_{2,\ell} := 2-\frac{a^{(0)}_1}{2} - \Delta w(\lambda_{\ell}) ,
\]
where we choose an integration contour $\Gamma_{\frac{n+4}{2}-\gamma}$ between the two poles, i.e.,
\[
 \frac{n}{2} +\frac{a^{(0)}_1}{2} -\Delta w(\lambda_{\ell}) < \gamma < \frac{n}{2} +\frac{a^{(0)}_1}{2} +\Delta w(\lambda_{\ell}) .
\] 
This particular choice is motivated by our considerations for the Laplace operator in Section \ref{Sec-Laplace}. Eventually our choice can be justified by explicit constructions of fundamental solutions and Green's functions from the parametrices. In the following, we will provide evidence  by considering selected examples for type-A differential operators. 

The parametrix $\Pop_0$ can be represented by an integral operator
\[
 \Pop_0\,u(r,\phi) = \int_{X} \int_{0}^{\infty} K_0(r,\phi|\tilde{r},\tilde{\phi}) u(\tilde{r},\tilde{\phi}) \tilde{r}^{n-1} d\tilde{r} \mu_{n-1}(\tilde{\phi}) d\tilde{\phi}
\]
with kernel function
\[
 K_0(r,\phi|\tilde{r},\tilde{\phi}) = -r^{a^{(0)}_1/2} \tilde{r}^{2-n-a^{(0)}_1/2}\sum_{\ell=0}^{\infty}
 \left( \frac{r_{<}}{r_{>}} \right)^{\Delta w(\lambda_{\ell})}
 \frac{p_{\ell}(\phi|\tilde{\phi})}{2\Delta w(\lambda_{\ell})}
\]

\begin{lemma}
Let $\tilde{\Aop}_0$ be an elliptic type-A differential operator in ${\cal C}^{\wedge}(X)$, i.e.,
 \begin{equation}
 	\tilde{\Aop}_0 = r^{-2} \left[ \biggl(-r \frac{\partial}{\partial r} \biggr)^{2} +a_1^{(0)} \biggl(-r \frac{\partial}{\partial r} \biggr) +a_0^{(0)} +b^{(0)} \Lambda_X \right] ,
 	\label{DefA20}
 \end{equation}
with coefficients satisfying 
\begin{equation}\label{Eq26111}
 a^{(0)}_0 =-a^{(0)}_1(n-2)-(n-2)^2 . 
\end{equation}
The distribution $\Pf k_0 \in {\cal D}'(\mathbb{R}^n)$, with
\begin{equation}\label{Eq26112}
k_0(r) := \lim_{\tilde{r} \rightarrow 0} K_0(r,\phi|\tilde{r},\tilde{\phi}) = -r^{2-n}\frac{p_0}{a_1^{(0)}+2(n-2)}
\end{equation}
is a fundamental solution of the corresponding differential operator $\Aop_{0}$ in $\mathbb{R}^n$, i.e.,
\[
 \Aop_{0} \Pf k_0 = \delta .
\]
\label{lemma2}
\end{lemma}
\begin{proof}
Under the first assumption in (\ref{Eq26111}) the limit (\ref{Eq26112}) follows immediatelly.
\ \\
\ \\
For any test function $g \in {\cal D}(\mathbb{R}^n)$ let $\tilde{g} :=\varphi^{\ast}g$ be its corresponding counterpart in $\tilde{\cal D}(C^{\wedge})$. Due to the spherical symmetry of $k_0$, we get
\begin{eqnarray*}
\lefteqn{\int_{\mathbb{R}^n} g(x) \Aop_0 \Pf k_0 \, dx} \\ 
 & = & \int_{0}^{\infty} r^{n-1} \frac{1}{r^2} \left[ \biggl( (n-2)+r \frac{\partial}{\partial r} \biggr)^2 \tilde{g}_{0}(r) +a^{(0)}_1 \biggl( (n-2)+r \frac{\partial}{\partial r} \biggr) \tilde{g}_{0}(r) +a^{(0)}_0 \tilde{g}_{0}(r) \right] \frac{k_0(r)}{p_0} dr
\end{eqnarray*}
where the right hand side only depends on the projection $\tilde{g}_0(r) := P_0 \tilde{g}$. Calculating the integrals 
\begin{eqnarray*}
\lefteqn{-\int_{0}^{\infty} r^{n-1} \frac{1}{r^2} \left( r \frac{\partial}{\partial r} \tilde{g}_{0}(r) \right) r^{2-n} \frac{1}{2\Delta w(\lambda_{0})} \, dr} \\
& = & -\int_{0}^{\infty} \partial_r \tilde{g}_{0}(r) \frac{1}{2\Delta w(\lambda_{0})} \, dr \\
& = & \lim_{r \rightarrow 0} \frac{\tilde{g}_{0}(r)}{2\Delta w(\lambda_{0})} \\
& = & \frac{g(0)}{2\Delta w(\lambda_{0})}
\end{eqnarray*}
and
\begin{eqnarray*}
\lefteqn{-\int_{0}^{\infty} r^{n-1} \frac{1}{r^2} \left( r \frac{\partial}{\partial r} \right) \left[ r \frac{\partial}{\partial r} \tilde{g}_{0}(r) \right] r^{2-n} \frac{1}{2\Delta w(\lambda_{0})} \, dr} \\
& = & -\int_{0}^{\infty} \partial_r \bigl( r \partial_r \tilde{g}_{0}(r) \bigr) \frac{1}{2\Delta w(\lambda_{0})} \, dr \\
& = & 0 ,
\end{eqnarray*}
we finally obtain
\[
 \int_{\mathbb{R}^n} g(x) \Aop_0 \Pf k_0 \, dx = g(0) . 
\]
\end{proof}

It remains to consider the higher order terms in the asymptotic parametrix construction for type-A differential operators.

\begin{proposition}
Parametrix symbols $h^{(-1)}_j(w,\lambda_{\ell})$ of higher order, i.e.~$j=1,2,\ldots$, which correspond to type-A differential operators have generically poles of order 1, however in particular cases second order poles may show up. For a symbol $h^{(-1)}_j(w,\lambda_{\ell})$ poles at most of order $2$ can appear only if 
\begin{equation}
 2\Delta w(\lambda_{\ell}) \in \{1,2,\ldots,j\} .
\label{dwcond}
\end{equation} 
If this is the case,   then $h^{(-1)}_j(w,\lambda_{\ell})$ potentially has $q_{\ell}:=j-k_{min}(\lambda_{\ell})+1$ poles of order 2 at
\[
 w_m(\lambda_{\ell}) = 2-\frac{a_1^{(0)}}{2}+m+\Delta w(\lambda_{\ell}), \quad m=0,1,\ldots,j-k_{min}(\lambda_{\ell}) ,
\]
with $k_{min}(\lambda_{\ell})= 2\Delta w(\lambda_{\ell})$.
\label{proposition3}
\end{proposition}
\begin{proof}
The symbols $h^{(-1)}_j(w,\lambda_{\ell})$ given by the recursive formula
(\ref{hm1p}) have zeros in the denominator, cf.~(\ref{hm1poles}), at
\[
 w_{1,m}(\lambda_{\ell})=2-\frac{a_1^{(0)}}{2}+m+\Delta w(\lambda_{\ell}), \quad w_{2,k}(\lambda_{\ell})=2-\frac{a_1^{(0)}}{2}+k-\Delta w(\lambda_{\ell}), \ \ m,k=0,\ldots,j .
\]
Therefore a zero of order 2 can appear only if 
\[
 k-m = 2\Delta w(\lambda_{\ell}) ,
\]
which leads to a pole of order 2 at $w_0:=w_{1,j}=w_{1,k}$ provided $w-w_0$ does not divide the corresponding nominator. So let $k_{min}(\lambda_l):= 2\Delta w(\lambda_{\ell})$, then $h^{(-1)}_j(w,\lambda_{\ell})$ potentially has $q_{\ell}=j-k_{min}(\lambda_l)+1$ poles of order 2 at
\[
 w_m = 2-\frac{a_1^{(0)}}{2}+m+\Delta w(\lambda_{\ell}), \quad m=0,1,\ldots,j-k_{min}(\lambda_{\ell}) .
\]
\end{proof}

Let us assume that the symbol $h^{(-1)}_j(w,\lambda_{\ell})$ has $2j+2-2q_{\ell}$, $q_{\ell} \geq 0$, simple poles,
which we arrange for each $\lambda_{\ell}$ in descending order, i.e.,
\[
 w_1(\lambda_{\ell}) > w_2(\lambda_{\ell}) > \cdots > w_{2p+2-2q_{\ell}}(\lambda_{\ell})
\]
where $w_1(\lambda_{\ell}),\ldots,w_{j_{\ell}}(\lambda_{\ell})$ lies right and $w_{j_{\ell}+1}(\lambda_{\ell}),\ldots,w_{2p+2}(\lambda_{\ell})$ left of the integration contour.
The remaining poles of order 2, denoted by $w^{(2)}_k(\lambda_{\ell})$, $k=1,\ldots,q_{\ell}$, are accordingly on the right side of the integration contour and are 
as well arranged in descending order, i.e.,
\[
 w^{(2)}_1(\lambda_{\ell}) > w^{(2)}_2(\lambda_{\ell}) > \cdots > w^{(2)}_{q_{\ell}}(\lambda_{\ell}) .
\]
The residue of such a poles is given by
\begin{eqnarray*}
 \Res_{w^{(2)}_k(\lambda_{\ell})} \left(\frac{r}{\tilde{r}} \right)^{-w} h^{(-1)}_j(w) & = & 
 \lim_{w \rightarrow w^{(2)}_k} \frac{d}{dw} \left[ \bigl( w-w^{(2)}_k(\lambda_{\ell}) \bigr)^2 \left(\frac{r}{\tilde{r}} \right)^{-w} h^{(-1)}_j(w) \right] \\
 & = & -\ln \left(\frac{r}{\tilde{r}} \right) \left(\frac{\tilde{r}}{r} \right)^{w^{(2)}_k} \Res_{w^{(2)}_k(\lambda_{\ell})}^{(1)} \bigl( h^{(-1)}_j(\cdot,\lambda_{\ell}) \bigr)
 +\left(\frac{\tilde{r}}{r} \right)^{w^{(2)}_k} \Res_{w^{(2)}_k(\lambda_{\ell})}^{(2)} \bigl( h^{(-1)}_j(\cdot,\lambda_{\ell}) \bigr)
\end{eqnarray*}
with
\[
 \Res_{w^{(2)}_k(\lambda_{\ell})}^{(1)} \bigl( h^{(-1)}_j(\cdot,\lambda_{\ell}) \bigr) := \lim_{w \rightarrow w^{(2)}_k(\lambda_{\ell})} \left[ \bigl( w-w^{(2)}_k(\lambda_{\ell}) \bigr)^2 h^{(-1)}_j(w) \right] ,
\]
\[
 \Res_{w^{(2)}_k(\lambda_{\ell})}^{(2)} \bigl( h^{(-1)}_j(\cdot,\lambda_{\ell}) \bigr) := \lim_{w \rightarrow w^{(2)}_k(\lambda_{\ell})} \frac{d}{dw} \left[ \bigl( w-w^{(2)}_k(\lambda_{\ell}) \bigr)^2 h^{(-1)}_j(w) \right] .
\]
Putting things together, we obtain
\begin{eqnarray*}
 \Pop_j u(r,\phi) & = & -\sum_{\ell=0}^{\infty} \sum_{k=1}^{j_{\ell}} \int_{X} \int_{0}^{r} \frac{\tilde{r}^{w_k(\lambda_{\ell})-n}}{r^{w_k(\lambda_{\ell})-2}}
 \bigl( \Res_{w_k(\lambda_{\ell})} h^{(-1)}_j(\cdot,\lambda_{\ell}) \bigr)  p_{\ell}(\phi|\tilde{\phi}) u(\tilde{r},\tilde{\phi}) \, \tilde{r}^{n-1} d\tilde{r} \mu_{n-1}(\tilde{\phi}) d\tilde{\phi} \\
 & & +\sum_{\ell=0}^{\infty} \sum_{k=j_{\ell}+1}^{2p+2-2q_{\ell}} \int_{X} \int_{r}^{\infty} \frac{\tilde{r}^{w_k(\lambda_{\ell})-n}}{r^{w_k(\lambda_{\ell})-2}}
 \bigl( \Res_{w_k(\lambda_{\ell})} h^{(-1)}_j(\cdot,\lambda_{\ell}) \bigr)  p_{\ell}(\phi|\tilde{\phi}) u(\tilde{r},\tilde{\phi}) \, \tilde{r}^{n-1} d\tilde{r} \mu_{n-1}(\tilde{\phi}) d\tilde{\phi} \\
  & & +\sum_{\ell=0}^{\infty} \sum_{k=1}^{q_{\ell}} \int_{X} \int_{0}^{r} \frac{\tilde{r}^{w_k^{(2)}(\lambda_{\ell})-n}}{r^{w_k^{(2)}(\lambda_{\ell})-2}}
 \left[ \ln \left(\frac{r}{\tilde{r}} \right) \Res_{w_k^{(2)}(\lambda_{\ell})}^{(1)} h^{(-1)}_j(\cdot,\lambda_{\ell}) \right. \\
 & & -\left. \Res_{w_k^{(2)}(\lambda_{\ell})}^{(2)} h^{(-1)}_j(\cdot,\lambda_{\ell}) \right] p_{\ell}(\phi|\tilde{\phi}) u(\tilde{r},\tilde{\phi}) \, \tilde{r}^{n-1} d\tilde{r} \mu_{n-1}(\tilde{\phi}) d\tilde{\phi}
\end{eqnarray*}

Let us summarize our calculations in the following lemma

\begin{lemma}
\label{lemmaPtypA}
For a type-A differential operator $\tilde{\Aop}$, the asymptotic parametrix given by the series
\[
 \Pop =  r^2 \sum_{j=0}^{\infty} r^j \opm_M^{\gamma-\frac{n+3}{2}} \bigl( h^{(-1)}_j(w,\Lambda_X) \bigr)
\]
can be represented by an integral operator
\[
 \Pop u(r,\phi) = \sum_{j=0}^{\infty} r^j \int_{X} \int_{0}^{\infty} K_j(r,\phi|\tilde{r},\tilde{\phi}) u(\tilde{r},\tilde{\phi}) \tilde{r}^{n-1} d\tilde{r} \mu_{n-1}(\tilde{\phi}) d\tilde{\phi}
\]
with kernel functions
\begin{eqnarray}
\nonumber
 K_j(r,\phi|\tilde{r},\tilde{\phi}) & = & -H(r-\tilde{r}) \sum_{\ell=0}^{\infty} \sum_{k=1}^{j_{\ell}} r^{2-w_k(\lambda_{\ell})} \tilde{r}^{w_k(\lambda_{\ell})-n} \bigl( \Res_{w_k(\lambda_{\ell})} h^{(-1)}_j(\cdot,\lambda_{\ell}) \bigr) p_{\ell}(\phi|\tilde{\phi}) \\ \nonumber
 & & +H(r-\tilde{r}) \sum_{\ell=0}^{\infty} \sum_{k=1}^{q_{\ell}} r^{2-w_k^{(2)}(\lambda_{\ell})} \tilde{r}^{w_k^{(2)}(\lambda_{\ell})-n} \left[ \ln \left(\frac{r}{\tilde{r}} \right) \Res_{w_k^{(2)}(\lambda_{\ell})}^{(1)} h^{(-1)}_j(\cdot,\lambda_{\ell}) \right. \\ \label{formulaKp}
 & & \left. -\Res_{w_k^{(2)}(\lambda_{\ell})}^{(2)} h^{(-1)}_j(\cdot,\lambda_{\ell}) \right] p_{\ell}(\phi|\tilde{\phi}) \\ \nonumber
  & & + H(\tilde{r} -r) \sum_{\ell=0}^{\infty} \sum_{k=j_{\ell}+1}^{2p+2-2q_{\ell}} r^{2-w_k(\lambda_{\ell})} \tilde{r}^{w_k(\lambda_{\ell})-n} \bigl( \Res_{w_k(\lambda_{\ell})} h^{(-1)}_j(\cdot,\lambda_{\ell}) \bigr) p_{\ell}(\phi|\tilde{\phi}) .
\end{eqnarray}
Here $H$ denotes the Heaviside function, i.e., $H(r-\tilde{r})=\begin{cases}
0,\quad r<\tilde{r}\\
1,\quad r>\tilde{r}
\end{cases}$.
\end{lemma}

Lemma \ref{lemma2} shows that for $a^{(0)}_0 =-a^{(0)}_1(n-2)-(n-2)^2$, the limit $\tilde{r} \rightarrow 0$ of $K_0$ yields a fundamental solution of $\Aop_0$. Under these conditions, we get $2\Delta w(\lambda_{0}) = a^{(0)}_1+2(n-2)$, which can lead, according to Proposition \ref{proposition3}, for integer $a^{(0)}_1$ to poles of order 2 in the asymptotic parametrix construction if (\ref{dwcond}) is satisfied. In such a case, due to the presence of logarithmic terms, it is not possible to take the limit $\tilde{r} \rightarrow 0$ of $K_j$ for $j \geq 2\Delta w(\lambda_{0})$. It should be emphasized, that this problem only concerns $\lambda_0$, because for higher eigenvalues, the prefactors $\tilde{r}^{w_k(\lambda_{\ell})-n}$, $\ell=1,2,\ldots$, have positive exponents and the terms $\tilde{r}^{w_k(\lambda_{\ell})-n} \ln(\tilde{r})$ vanish in the limit $\tilde{r} \rightarrow 0$.
Furthermore it might happen, that poles are located on the real axis between the integration contour $\Gamma_{\frac{n+4}{2}-\gamma}$ and the pole $w_{j_n}(\lambda_0)=n$. According to our construction, we only assume that $\Gamma_{\frac{n+4}{2}-\gamma}$ cuts the real axis between the poles $w_{j_{n-1}}(\lambda_0)=4-a^{(0)}_1-n$ and $w_{j_n}(\lambda_0)=n$, so depending on the choice of $\gamma$ such cases might occur.    

For $r > \tilde{r}$, we decompose the kernel according to
\[
 K_j(r,\phi|\tilde{r},\tilde{\phi}) = K^{0)}_j(r,\phi|\tilde{r},\tilde{\phi}) + K^{(1)}_j(r,\phi|\tilde{r},\tilde{\phi}) \quad (r > \tilde{r})
\]
with 
\begin{eqnarray*}
 K^{(0)}_j(r,\phi|\tilde{r},\tilde{\phi}) & := & -\sum_{\ell=0}^{\infty} \sum_{k=1}^{j_{\ell}} H_{0} \bigl( w_k(\lambda_{\ell})-n \bigr) r^{2-w_k(\lambda_{\ell})} \tilde{r}^{w_k(\lambda_{\ell})-n} \bigl( \Res_{w_k(\lambda_{\ell})} h^{(-1)}_j(\cdot,\lambda_{\ell}) \bigr) p_{\ell}(\phi|\tilde{\phi}) \\ 
 & & +\sum_{\ell=0}^{\infty} \sum_{k=1}^{q_{\ell}} r^{2-w_k(\lambda_{\ell})} \tilde{r}^{w_k(\lambda_{\ell})-n} \left[ \ln(r) \Res_{w_k^{(2)}(\lambda_{\ell})}^{(1)} h^{(-1)}_j(\cdot,\lambda_{\ell}) \right. \\  & & \left. -\Res_{w_k^{(2)}(\lambda_{\ell})}^{(2)} h^{(-1)}_j(\cdot,\lambda_{\ell}) \right] p_{\ell}(\phi|\tilde{\phi})
\end{eqnarray*}
\begin{eqnarray*}
 K^{(1)}_j(r,\phi|\tilde{r},\tilde{\phi}) & := & -\sum_{\ell=0}^{\infty} \sum_{k=1}^{j_{\ell}} H_{1} \bigl( n-w_k(\lambda_{\ell}) \bigr) r^{2-w_k(\lambda_{\ell})} \tilde{r}^{w_k(\lambda_{\ell})-n} \bigl( \Res_{w_k(\lambda_{\ell})} h^{(-1)}_j(\cdot,\lambda_{\ell}) \bigr) p_{\ell}(\phi|\tilde{\phi}) \\
  & & -\sum_{\ell=0}^{\infty} \sum_{k=1}^{q_{\ell}} r^{2-w_k(\lambda_{\ell})} \tilde{r}^{w_k(\lambda_{\ell})-n} \ln(\tilde{r}) \Res_{w_k^{(2)}(\lambda_{\ell})}^{(1)} h^{(-1)}_j(\cdot,\lambda_{\ell}) p_{\ell}(\phi|\tilde{\phi}) ,
\end{eqnarray*}
where the Heavyside functions are defined at $0$ by $H_{0}(0)=1$, $H_{1}(0)=0$, respectively.
Taking the limit $\tilde{r} \rightarrow 0$ of the average of $K^{0)}_j$, we obtain 
\begin{eqnarray}
\nonumber
 k_j(r) := \lim_{\tilde{r} \rightarrow 0} K^{0)}_j(r,\phi|\tilde{r},\tilde{\phi}) & = & -r^{2-n} \sum_{k=1}^{j_{0}} \delta \bigl( w_k(\lambda_{0})-n \bigr) \bigl( \Res_{w_k(\lambda_{0})} h^{(-1)}_j(\cdot,\lambda_{0}) \bigr) p_{0} \\ \nonumber
 & & +r^{2-n}\sum_{k=1}^{q_{0}} \delta \bigl( w_k^{(2)}(\lambda_{0})-n \bigr) \left[ \ln(r) \Res_{w_k^{(2)}(\lambda_{0})}^{(1)} h^{(-1)}_j(\cdot,\lambda_{0}) \right. \\ \label{kp}
 & & \left. -\Res_{w_k^{(2)}(\lambda_{0})}^{(2)} h^{(-1)}_j(\cdot,\lambda_{0}) \right] p_{0}
\end{eqnarray}

\begin{remark}
Logarithmic terms can only appear for $j \geq 2\Delta w(\lambda_{0})$.
For each order $j$ of the asymptotic parametrix contruction, there are only a finite number of non vanishing terms in the sum over $\ell$. Furthermore, there is at most only one nonvanishing term on the right hand side of (\ref{kp}).  
\end{remark}

\section{Shifted Laplace operators}
\label{shiftedLaplace}
As a first example for the asymptotic parametrix construction, we consider the shifted Laplace operator $\tilde{\Delta}_n -\kappa^2$ in dimension $n \geq 3$ with $\tilde{\Delta}_n$ given by (\ref{Lspher}). 
The shifted Laplace operator is of type-A with base $X=S^{n-1}$ and $\Lambda_X=-\Delta_{S^{n-1}}$.
Its Mellin symbol has the form
\[
 h(w)=h_0(w) +r^2h_2(w)
\]
with $h_0$ and $h_2$ given by (\ref{mL}) and $-\kappa^2$, respectively. Application of the recursion formula (\ref{hm1p}) yields
\begin{equation}
 h^{(-1)}_{2m}(w) = \kappa^{2m} \prod_{j=0}^m h^{(-1)}_0(w-2j) \quad \mbox{and} \quad h^{(-1)}_{2m+1}=0 \quad \mbox{for} \ m=0,1,\ldots ,
\label{h2m}
\end{equation}
where $h^{(-1)}_0$ is given by (\ref{poles}). 

\subsection{Fudamental solutions from the parametrix}
To exemplify our approach, we only consider the cases $n=3,4$ in the following proposition. However it is clear that analogous calculations can be performed in any dimension $\geq 3$. 

\begin{proposition}
For $\Delta_n -\kappa^2$, $n=3,4$, the asymptotic parametrix construction, with integration contour $\Gamma_{\frac{n+4}{2}-\gamma}$ taken for $2-\frac{n}{2} < \gamma < 3-\frac{n}{2}$, yields a fundamental solution
\[
 u = \Pf k(r) \ \mbox{with} \  k(r) := k_0(r) +rk_1(r) +r^2k_2(r) \cdots ,
\]
where the functions $k_p$, $p=0,1,\ldots$, are given by (\ref{kp}).
\end{proposition}

\begin{proof}
According to (\ref{h2m}), the symbols $h^{(-1)}_{2m,0}$ have a pole at $n$ if $2+2j$ or $n+2j$ equals $n$ for some $j \in \{0,1,\ldots,m\}$. Therefore for $n$ odd, only poles of order 1 appear, whereas for $n$ even, poles of order 1 and 2 show up. For an admissible $\gamma$, no poles appear on the strip between the integration contour and the pole at $n$. 

\paragraph{Case 1: $n$ odd} \ \\
In each symbol $h^{(-1)}_{2m}$, there is only for $j=0$ a pole of order 1 at $n$. Formulas (\ref{kp}) and (\ref{h2m}) yield
\[
 k_{2m}(r) = -r^{2-n} \frac{(-1)^m\kappa^{2m}p_0}{2^m m!} \prod_{j=0}^m \frac{1}{n-2j-2}, \ 
 \mbox{with} \ p_0=\frac{(n-2)!!}{2^{\frac{n+1}{2}}\pi^{\frac{n-1}{2}}} .
\] 
In particular, for $n=3$ we get
\begin{equation}
 k_{2m}(r) = -r^{-1} \frac{\kappa^{2m}}{4\pi(2m)!}
\label{k2m3d}
\end{equation}
and
\[
 k(r) = \sum_{j=0}^{\infty} r^j k_j(r) = -\frac{1}{4\pi r} \sum_{m=0}^{\infty} \frac{(\kappa r)^{2m}}{(2m)!}= -\frac{1}{4\pi r} \cosh(\kappa r)
\]
This can be compared with the standard fundamental solution, cf.~\cite{Schwartz}, 
\begin{equation}
 u(x) = \Pf -\frac{1}{4\pi r} e^{-\kappa r} . 
\label{un3}
\end{equation}
It can be easily seen, the difference between both fundamental solutions
\[
 \Pf -\frac{1}{4\pi r} \bigl( \cosh(\kappa r) - e^{-\kappa r} \bigr) = \Pf -\frac{1}{4\pi r} \sinh(\kappa r)
\] 
correspond to a smooth harmonic function in $\mathbb{R}^3$.

\paragraph{Case 2: $n$ even} \ \\
For $p<n-2$, the symbols $h^{(-1)}_{2m}$ have a pole of order 1 and for $p \geq n$ a pole of order 2 at $n$.
Let
\[
 Q_0(w) := \frac{1}{w-2} \ \ \mbox{and} \ \  
 Q_m(w) := \prod_{j=0, 2j+2 \neq n}^{m} \frac{1}{w-2j-2} \prod_{j=1}^{m} \frac{1}{w-2j-n} \quad \mbox{for} \ m \geq 1 ,
\]
as well as
\[
 Q'(w)=\frac{d Q_m}{dw}(w) = - Q_m(w) \left[ \sum_{j=0, 2j+2 \neq n}^{m} \frac{1}{w-2j-2} +\sum_{j=1}^{m} \frac{1}{w-2j-n} \right]
 \quad \mbox{for} \ m \geq 1 .
\]
With this, we get
\[
 \Res_{n} h^{(-1)}_{2m} = \kappa^{2m} Q_m(n) \quad \mbox{for} \ 2m < n-2 , 
\]
and 
\[
 \Res_{n}^{(1)} h^{(-1)}_{2m} = \kappa^{2m} Q_m(n) , \quad
 \Res_{n}^{(2)} h^{(-1)}_{2m} = \kappa^{2m} Q'_m(n) \quad \mbox{for} \ 2m \geq n-2 .
\]
Therefore, according to formulas (\ref{kp}) and (\ref{h2m}),
\[
 k_{2m}(r) = -r^{2-n} \kappa^{2m} \left\{ \begin{array}{cc} Q_m(n) p_{0} & \mbox{for} \ 2m < n-2 \\
 -\bigl( \ln(r) Q_m(n) -Q'_m(n) \bigr) p_{0} & \mbox{for} \ 2m \geq n-2 \end{array} \right.
\]
Let us demonstrate the proposition by an explicit calculation for $n=4$, where we have
\[
 Q_0(4)=\frac{1}{2}, \quad Q_m(4) = -\frac{1}{2^{2m} m! (m-1)!} \quad \mbox{for} \ m \geq 1
\]
and 
\[
 Q'_m(4) = -\frac{1}{2^{2m+1} m! (m-1)!} \left( -1 +\sum_{j=1}^{m-1} j^{-1} +\sum_{j=1}^{m} j^{-1} \right) \quad \mbox{for} \ m \geq 1 .
\]
Finally, with $p_0=\frac{1}{2\pi^2}$, we get
\[
 k(r) = -\frac{1}{4\pi^2} r^{-2}  -\sum_{m=1}^{\infty} r^{-2+2m} \frac{\kappa^{2m}}{2\pi^2} \left[ \frac{1}{2^{2m} m! (m-1)!} \ln(r) -\frac{1}{2^{2m+1} m! (m-1)!} \left( -1 +\sum_{j=1}^{m-1} j^{-1} +\sum_{j=1}^{m} j^{-1} \right) \right] 
\]
and after some algebraic manipulations, we obtain
\begin{equation}
 k(r) = -\frac{1}{4\pi^2} r^{-2}  -\frac{1}{4\pi^2} (\kappa r)^{-1} I_1(\kappa r) \ln(r) + \frac{\kappa^2}{(4\pi)^2} \sum_{k=0}^{\infty} \frac{1}{2^{2k} k! (k+1)!} \left( -1 +\sum_{j=1}^{k} j^{-1} +\sum_{j=1}^{k+1} j^{-1} \right) (\kappa r)^{2k}
\label{krLk}
\end{equation}
where we used, cf.~\cite{AS}[Eq.~9.6.10],
\[
 (\kappa r)^{-1} I_1(\kappa r) = \frac{1}{2} \sum_{k=0}^{\infty} \frac{1}{2^{2k}k!(k+1)!} (\kappa r)^{2k} ,
\]
which according to (\ref{modu}), with $\alpha =-1$, represents a fundamental solution.
\end{proof}

Let us finally consider the role of the terms in the kernel of the parametrix which have been excluded from our construction, i.e.,
\[
 K^{(1)}(r,\phi|\tilde{r},\tilde{\phi}) = \sum_0^{\infty} r^p K^{(1)}_p(r,\phi|\tilde{r},\tilde{\phi}) .
\]

\begin{remark}
According to our discussion after Lemma \ref{lemmaPtypA}, it is only the $\ell=0$ term 
for which the Limit $\tilde{r} \rightarrow 0$ cannot be performed.

For $n$ even, $k^{(1)}_{2m} := P_0 K^{(1)}_{2m}$, $m=1,2,\ldots$, is given by 
\begin{eqnarray*}
  k^{(1)}_{2m}(r,\tilde{r}) & := &
  -r^{2-n} \sum_{j=1}^{q_{0}} \delta \bigl( w_j^{(2)}(\lambda_{0})-n \bigr) \ln(\tilde{r}) \left( \Res_{w_j^{(2)}(\lambda_{0})}^{(1)} h^{(-1)}_{2m}(\cdot,\lambda_{0}) \right) p_{0} \\
  & = & -r^{2-n} \ln(\tilde{r}) \kappa^{2m} Q_m(n) p_{0}
\end{eqnarray*}
which sums up to  
\[
 k^{(1)}(r,\tilde{r}) := \sum_{m=1}^{\infty} r^{2m}k^{(1)}_{2m}(r,\tilde{r}) = -\left( \sum_{m=1}^{\infty} r^{2(m+1)-n} \kappa^{2m} Q_m(n) \right) \ln(\tilde{r}) p_{0} ,
\]
in particular for $n=4$, we get
\begin{eqnarray*}
 k^{(1)}(r,\tilde{r}) & = & \left( \sum_{m=1}^{\infty} r^{2m-2} \kappa^{2m} \frac{1}{2^{2m} m! (m-1)!} \right) \ln(\tilde{r}) p_{0} \\
 & = & \frac{\kappa^2}{4} \left( \sum_{k=0}^{\infty} (\kappa r)^{2k} \frac{1}{2^{2k} (k+1)! k!} \right) \ln(\tilde{r}) p_{0} \\
 & = & \frac{1}{2} r^{-1} I_1(\kappa r) \ln(\tilde{r}) p_{0} 
\end{eqnarray*}
In the second line, we used (\ref{rI}), which shows that $k^{(1)}$ corresponds to the kernel of a Green operator which maps onto a smooth function in the kernel of the shifted Laplace operator.
\end{remark}

\subsection{Green's functions from the parametrix}
In order to demonstrate that one can even get a generalized Laplace expansion of a Green's function from the kernel of the parametrix, we consider the case $n=3$ where no logarithmic terms show up in the expansion.

\begin{proposition}
For $\Delta_3-\kappa^2$ ,the kernel of the asymptotic parametrix  
\[
 K_{\kappa}(r,\phi|\tilde{r},\tilde{\phi}) = \sum_0^{\infty} r^p K_p(r,\phi|\tilde{r},\tilde{\phi}) .
\]
is given by
\begin{eqnarray}
\nonumber
 K_{\kappa}(r,\phi|\tilde{r},\tilde{\phi}) & = & H(r-\tilde{r}) \frac{\pi}{2} \sum_{\ell=0}^{\infty} (-1)^{\ell} \left[ \frac{I_{\ell+\frac{1}{2}}(\kappa r)}{\sqrt{r}} \frac{I_{-\ell-\frac{1}{2}}(\kappa \tilde{r})}{\sqrt{\tilde{r}}} - \frac{I_{-\ell-\frac{1}{2}}(\kappa r)}{\sqrt{r}} \frac{I_{\ell+\frac{1}{2}}(\kappa \tilde{r})}{\sqrt{\tilde{r}}} \right] p_{\ell}(\phi,\tilde{\phi}) \\ \label{Kkappa}
 & & -\sqrt{\frac{\pi\kappa}{2}} \sum_{\ell=0}^{\infty} \frac{I_{\ell+\frac{1}{2}}(\kappa r)}{\sqrt{r}} S_{\ell}(\kappa \tilde{r}) p_{\ell}(\phi,\tilde{\phi})
\end{eqnarray}
with
\[
 S_{\ell}(\kappa \tilde{r}) := \sum_{m=0}^{\lfloor \ell/2 \rfloor} (\kappa \tilde{r})^{-1-\ell+2m} \frac{(-1)^m\bigl( 2(\ell-m)-1 \bigr)!!}{2^m m!!} .
\]
and $I_{\pm\ell\pm\frac{1}{2}}$ denote modified Bessel function of first order.
The corresponding regular family of distributions 
\[
 K_{\kappa}( \cdot, \tilde{x}) := \left\{ \begin{array}{cc}  \Pf K_{\kappa}(\cdot|\tilde{r},\tilde{\phi}) & \mbox{for} \ \tilde{x} = \varphi(\tilde{r},\tilde{\phi}) \neq 0 \\
 \lim_{\tilde{r} \rightarrow 0} \Pf K_{\kappa}(\cdot|\tilde{r},\tilde{\phi}) & \mbox{for} \ \tilde{x} =0 \end{array} \right.
\]
satisfies Definition \ref{greendef} for a Green's function.
\end{proposition}
\begin{proof}
Details of the calculation for (\ref{Kkappa}) are given in Appendix \ref{appendixKkappa}. 
It is intructive to verify (\ref{Kkappa}) by comparison with the canonical Green's
function of the shifted Laplace operator, cf.~\cite{Schwartz}, which can be obtained from the fundamental solution (\ref{un3})
\[
 G_{\kappa}(\cdot,\tilde{x}) = -\frac{1}{4\pi |\cdot-\tilde{x}|} e^{-\kappa |\cdot-\tilde{x}|}
\]
For this Green's function a generalized Laplace expansion based on Gegenbauers addition theorem is known in the literature, cf.~\cite{Watson}[p.366], such that
\[
 G_{\kappa}( \cdot, \tilde{x}) := \left\{ \begin{array}{cc}  \Pf G_{\kappa}(\cdot|\tilde{r},\tilde{\phi}) & \mbox{for} \ \tilde{x} = \varphi(\tilde{r},\tilde{\phi}) \neq 0 \\
 \lim_{\tilde{r} \rightarrow 0} \Pf G_{\kappa}(\cdot|\tilde{r},\tilde{\phi}) & \mbox{for} \ \tilde{x} =0 \end{array} \right.
\]
with
\begin{eqnarray}
\nonumber
 G_{\kappa}(r,\phi|\tilde{r},\tilde{\phi})
  & = & -H(r-\tilde{r})\sum_{\ell=0}^{\infty} \frac{K_{\ell+\frac{1}{2}}(\kappa r)}{\sqrt{r}} \frac{I_{\ell+\frac{1}{2}}(\kappa \tilde{r})}{\sqrt{\tilde{r}}} p_{\ell}(\phi,\tilde{\phi}) \\ \label{Gkappa}
 & & -H(r-\tilde{r})\sum_{\ell=0}^{\infty} \frac{I_{\ell+\frac{1}{2}}(\kappa r)}{\sqrt{r}} \frac{K_{\ell+\frac{1}{2}}(\kappa \tilde{r})}{\sqrt{\tilde{r}}} p_{\ell}(\phi,\tilde{\phi})
\end{eqnarray}
where
\begin{equation}
\label{Besselsk}
 K_{\ell+\frac{1}{2}}(\kappa \tilde{r}) = \frac{\pi}{2} (-1)^{\ell} \left( I_{-\ell-\frac{1}{2}}(\kappa r) -I_{\ell+\frac{1}{2}}(\kappa r) \right)
\end{equation}
denotes a modified Bessel function of second order.

Let us consider the differences between the kernel functions (\ref{Kkappa}) and (\ref{Gkappa}), which is given by
\begin{equation}
 K_{\kappa}(r,\phi|\tilde{r},\tilde{\phi}) -G_{\kappa}(r,\phi|\tilde{r},\tilde{\phi})
 = \sum_{\ell=0}^{\infty} \frac{I_{\ell+\frac{1}{2}}(\kappa r)}{\sqrt{r}} \left( \frac{ K_{\ell+\frac{1}{2}}(\kappa \tilde{r})}{\sqrt{\tilde{r}}} +\sqrt{\frac{\pi\kappa}{2}} S_{\ell}(\kappa \tilde{r}) \right) p_{\ell}(\phi,\tilde{\phi}) .
\end{equation}
According to (\ref{rI}), the distribution corresponding to this difference maps ${\cal D}(\mathbb{R}^3)$ into a smooth function in the kernel of $\Delta_3-\kappa^2$.
\end{proof}

Although, the kernel of the parametrix (\ref{Kkappa}) satisfies our definition of a Green's function it is clear, by comparison with (\ref{Gkappa}), that it has some shortcomings. First of all, (\ref{Kkappa}) is not a symmetric kernel function, unlike (\ref{Gkappa}), which reflects the fact that the shifted Laplace operator is symmetric.
Given (\ref{Kkappa}), this can be cured however by a simple consideration.
First of all, the second sum can be skipped, bcause it maps ${\cal D}(\mathbb{R}^3)$ into smooth functions in the kernel of $\Delta_3-\kappa^2$. Furthermore, according to the previous remark, we are free to add a term of the form
\[
 \frac{\pi}{2} (-1)^{\ell} \sum_{\ell=0}^{\infty} \frac{I_{\ell+\frac{1}{2}}(\kappa r)}{\sqrt{r}} 
 \frac{A(\kappa \tilde{r})}{\sqrt{\tilde{r}}} p_{\ell}(\phi,\tilde{\phi})  ,
\]
in order to restore symmetry under permutation $(r,\phi) \leftrightarrow (\tilde{r},\tilde{\phi})$. Permutational symmetry requires
\[
 - \frac{I_{-\ell-\frac{1}{2}}(\kappa r)}{\sqrt{r}} \frac{I_{\ell+\frac{1}{2}}(\kappa \tilde{r})}{\sqrt{\tilde{r}}} +\frac{I_{\ell+\frac{1}{2}}(\kappa r)}{\sqrt{r}} \left( \frac{I_{-\ell-\frac{1}{2}}(\kappa \tilde{r})}{\sqrt{\tilde{r}}} +\frac{A(\kappa \tilde{r})}{\sqrt{\tilde{r}}} \right) = \frac{A(\kappa r)}{\sqrt{r}} \frac{I_{\ell+\frac{1}{2}}(\kappa r)}{\sqrt{r}}
\]
which is obviously satisfied by
\[
 A(\kappa \tilde{r}) = -I_{-\ell-\frac{1}{2}}(\kappa \tilde{r}) \mod I_{\ell+\frac{1}{2}}(\kappa \tilde{r})
\]
The remaining ambiguity can be resolved by
another shortcomming of (\ref{Kkappa}) which concerns its asymptotic behaviour for $r,\tilde{r} \rightarrow \infty$. Whereas (\ref{Gkappa}) decays exponentially, this is not the case for (\ref{Kkappa}). The reason is due to the fact that the modified Bessel functions of the first kind $I_{\pm\ell\pm\frac{1}{2}}$ increase exponentially, and it is only their difference, the modified Bessel functions of the second kind $K_{\ell+\frac{1}{2}}$, cf.~(\ref{Besselsk}), which decays exponentially \cite{AS}[9.7.1,9.7.2]. In order to get the desired asymptotic behaviour, we have to add the term
\[
 \frac{\pi}{2} (-1)^{\ell} \sum_{\ell=0}^{\infty} \frac{I_{\ell+\frac{1}{2}}(\kappa r)}{\sqrt{r}} 
 \frac{I_{\ell+\frac{1}{2}}(\kappa \tilde{r})}{\sqrt{\tilde{r}}} p_{\ell}(\phi,\tilde{\phi})  ,
\]
which obviously preserves permutational symmetry. With these modification (\ref{Kkappa}) ultimatey becomes (\ref{Gkappa}). From the point of view of singular analysis the required modifications can be assigned to Green operators, which is linked to the fact that these terms map into spaces with fixed asymptotic behaviour.

\section{Application in Physics: Scattering theory}
\label{scatteringtheory}
In pseudo-differential calculus the notion of ellipticity  poses severe restrictions on the partial differential operators for which a parametrix exists and therefore excludes many interesting cases. 
A typical example, already mentioned in Section \ref{Sec-Laplace}, is the Laplace operator which does not satisfy the exit condition. As a possible 
remedy, one can consider the elliptic shifted Laplace operator, calculate a fundamental solution or Green's function and finally takes the limit $\kappa \rightarrow 0$, which yields the corresponding quantities for the Laplace operator, cf.~Appendix \ref{app-regfs} for further details. 

Such an approach can be easily generalized by performing analytic continuations with respect to a given parameter on which the fundamental solution or Green's function depends. The method of analytical continuation paves the way to many interesting applications in physics, in particular quantum many-body and scattering theory. 

\begin{remark}
Typical differential operators in scattering theory do not satisfy the exit condition for ellipticity. A prominent example is the Helmholtz equation
\[
 \bigl( \Delta +\kappa^2 \bigr) u=\delta
\]
which differs from the shifted Laplace equation, discussed in Section \ref{shiftedLaplace}, by the minus sign in front of constant $\kappa^2$. Actually, the two equations and their fundamental solutions are related by analytic continuation of the parameter $\kappa$, i.e., a rotation in the complex plane by $\pm \pi/2$ transforms one into the other. More explicitly, let us consider the fundamental solution of the Helmholtz equation in $n$-dimension, cf.~(\cite{EZ18}),
\[
 u_{n,\kappa} =\Pf \frac{1}{4i(2\pi)^{\frac{n-2}{2}}} k^{\frac{n-2}{2}} r^{\frac{2-n}{2}} H^{(1)}_{\frac{n-2}{2}}(\kappa r) ,
\] 
where $H^{(1)}_{\frac{n-2}{2}}$ denotes a Hankel function of the first kind. Using the indentity, cf.~\cite{AS}[9.6.4],
\[
 H^{(1)}_{\frac{n-2}{2}}(\kappa r) = -\frac{2i}{\pi} e^{-\frac{1}{2}\frac{n-2}{2}\pi i} K_{\frac{n-2}{2}}(-i\kappa r) ,
\]
the fundamental solution becomes
\[
 u_{n,\kappa} =\Pf -(2\pi)^{-\frac{n}{2}} (-i\kappa)^{\frac{n-2}{2}} r^{\frac{2-n}{2}} K_{\frac{n-2}{2}}(-i\kappa r)
\]
which can be obtained from the standard fundamental solution of the shifted Laplacian, cf.~(\ref{fndsolLk}) by rotating $\kappa$ from the real to the negative imaginary axis. The renormalized fundamental solution (\ref{krLk}), which we have obtained from the asymptotic parametrix construction, differs from it by holomorpic terms only. Therefor it is possible to get access to a fundamental solution of the Helmholtz equation by an analytic continuation from the asymptotic parametrix construction.  
\end{remark}

Analytic continuation, like the one discussed in the previous remark is a versatile tool in quantum theory \cite{RS3}. A rigorous treatment however, requires full control on the convergence of the power series expansion of the    
asymptotic parametrix construction. We want to discuss this topic in the simplest non trivial case
of $n=3$ nonrelativistic electron-nucleus scattering based on a differential operator, given in spherical coordinates by 
\begin{equation}
 \tilde{\Aop} := -2 \bigl( \Hop +\kappa^2 \bigr) = \frac{1}{r^2} \biggl[ \biggl(-r \frac{\partial}{\partial r} \biggr)^2 -\biggl(-r \frac{\partial}{\partial r} \biggr) +\Delta_{S^2} +r2Z - r^2 2\kappa^2 \biggr] ,
\label{centralpotential}
\end{equation}
where $\Hop$ denotes the non relativistic Hamiltonian of an electron (atomic units), interacting with a nucleus of charge $Z$, cf.~\cite{MM65} for further details. 
The corresponding Mellin symbol $h$ of the differential operator $\Aop$ is given by the expansion
\[
 h(w)=h_0(w)+rh_1+r^2h_2 \quad \mbox{with} \ h_0(w)=w^2-w+\Delta_{S^2}, \ h_1=2Z, \ h_2=-2\kappa^2 .
\]
According to (\ref{hm1p}), we get for the asymptotic parametrix symbol of order 1
\[
 h_{1}^{(-1)}(w)=-\bigl(h_{0}^{(-1)}(w-1) \bigr) Zh_0^{(-1)}(w) \quad \mbox{with} \ h_0^{(-1)}(w)=\frac{1}{h_0(w-2)}
\]
and for order $p \geq 2$, the recursive formula 
\begin{equation}
 h_{p}^{(-1)}(w)=-\biggl( \bigl(h_{p-1}^{(-1)}(w-1) \bigr) Z-\bigl(h_{p-2}^{(-1)}(w-2) \bigr)\kappa^2 \biggr)h_0^{(-1)}(w)  
\label{scatiter}
\end{equation}
In order to keep track of the combinatorics, let us introduce the notion of binary words given by the following definition
\begin{definition}  
A binary word $\alpha$ is composed of the numbers 1 and 2, i.e., $\alpha =(\alpha_1,\alpha_2,\alpha_3,\ldots)$ with $\alpha_i \in \{1,2\}$. For a word composed of $N^{\alpha}$ characters, let $N_Z^{\alpha}$ and $N_{\kappa}^{\alpha}$ denote the number of characters 1 and 2, respectively.
Furthemore, let $S_p$, $p \geq 1$, denote the set of all possible words such that $p=\sum_{k=1}^{N^{\alpha}} \alpha_k=N_Z^{\alpha}+2N_{\kappa}^{\alpha}$.
\label{binaryword} 
\end{definition}  
According to the previous definition, we have, e.g., $S_1 =\{ (1) \}$, $S_2 =\{ (1,1), (2) \}$, $S_3 =\{ (1,1,1), (1,2), (2,1) \}$.
\begin{proposition}
Given the word operations $\lfloor_1$, $\lfloor_2$ $\bigl( _1\rfloor$, $_2\rfloor \bigr)$, given by $\alpha =(\hdots) \rightarrow \alpha\lfloor_1 = (\hdots,1)$ $\bigl( \alpha =(\hdots) \rightarrow _1\rfloor\alpha = (1,\hdots) \bigr)$ and $\alpha =(\hdots) \rightarrow \alpha\lfloor_2 = (\hdots,2)$ $\bigl( \alpha =(\hdots) \rightarrow _2\rfloor\alpha = (2,\hdots) \bigr)$, respectively.
With these operations acting on all elements of a set of binary words, we get $S_p=S_{p-1}\lfloor_1 \cup S_{p-2}\lfloor_2$ and $S_p= _1\rfloor S_{p-1} \cup _2\rfloor S_{p-2}$, respectively.
\label{Spfloor}
\end{proposition}
\begin{proof}
$S_p=S_{p-1}\lfloor_1 \cup S_{p-2}\lfloor_2$ follows from the fact that $\alpha \in S_{p}$ is either of the form $(\hdots,1)$ or of the form $(\hdots,2)$. The other one follows in an analogous manner.
\end{proof}
\begin{proposition}
The $p$'th order symbol of the asymptotic paramtrix is given by
\begin{equation}
 h_{p}^{(-1)}(w) = \sum_{\alpha \in S_p} (-1)^{N_Z^{\alpha}} 2^{N^{\alpha}} Z^{N_Z^{\alpha}} \bigl( \kappa^2 \bigr)^{N_{\kappa}^{\alpha}} \prod_{j=0}^{N^{\alpha}} h_0^{(-1)}(w-p_j) ,
\label{hm1pword}
\end{equation}
with $p_0=p$ and $p_j=p-\sum_{k=1}^j \alpha_k$ for $j \geq 1$.
\end{proposition}
\begin{proof}
The proof goes by induction. Formula (\ref{hm1pword}) is obviously true for $p=1,2$. Using (\ref{scatiter}) and Proposition \ref{Spfloor}, we can perform the induction step
\begin{eqnarray*}
 h_{p+1}^{(-1)}(w) & = & \sum_{\alpha \in S_p} (-1)^{N_Z^{\alpha}+1} 2^{N^{\alpha}+1} Z^{N_Z^{\alpha}+1} \bigl( \kappa^2 \bigr)^{N_{\kappa}^{\alpha}} \left( \prod_{j=0}^{N^{\alpha}}  h_0^{(-1)}(w-p_j-1) \right) h_0^{(-1)}(w) \\
 &  & + \sum_{\alpha \in S_{p-1}} (-1)^{N_Z^{\alpha}} 2^{N^{\alpha}+1} Z^{N_Z^{\alpha}} \bigl( \kappa^2 \bigr)^{N_{\kappa}^{\alpha}+1} \left( \prod_{j=0}^{N^{\alpha}} h_0^{(-1)}(w-p_j-2) \right) h_0^{(-1)}(w) \\ 
 & = & \sum_{\alpha \in S_p\lfloor_1} (-1)^{N_Z^{\alpha}} 2^{N^{\alpha}} Z^{N_Z^{\alpha}} \bigl( \kappa^2 \bigr)^{N_{\kappa}^{\alpha}} \prod_{j=0}^{N^{\alpha}}  h_0^{(-1)}(w-p_j) \\
 &  & + \sum_{\alpha \in S_{p-1}\lfloor_2} (-1)^{N_Z^{\alpha}} 2^{N^{\alpha}} Z^{N_Z^{\alpha}} \bigl( \kappa^2 \bigr)^{N_{\kappa}^{\alpha}} \prod_{j=0}^{N^{\alpha}}  h_0^{(-1)}(w-p_j) \\
 & = & \sum_{\alpha \in S_{p+1}} (-1)^{N_Z^{\alpha}} 2^{N^{\alpha}} Z^{N_Z^{\alpha}} \bigl( \kappa^2 \bigr)^{N_{\kappa}^{\alpha}} \prod_{j=0}^{N^{\alpha}}  h_0^{(-1)}(w-p_j) .
\end{eqnarray*}
\end{proof}

Let us first consider the physically most relevant case $n=3$ and calculate the spherical limit $k_p$, given by (\ref{kp}), i.e.,
\begin{eqnarray}
\label{kp3a}
 k_p(r) & = & -\frac{1}{4\pi r} \sum_{\alpha \in S_{p-2}\lfloor_2} (-1)^{N_Z^{\alpha}} 2^{N^{\alpha}} Z^{N_Z^{\alpha}} \bigl( \kappa^2 \bigr)^{N_{\kappa}^{\alpha}} \bigl( \Res_{3} \prod_{j=0}^{N^{\alpha}}  h_0^{(-1)}(\cdot-p_j,\lambda_0) \bigr) \\ \nonumber
 &  & +\frac{1}{4\pi r} \sum_{\alpha \in S_{p-1}\lfloor_1} (-1)^{N_Z^{\alpha}} 2^{N^{\alpha}} Z^{N_Z^{\alpha}} \bigl( \kappa^2 \bigr)^{N_{\kappa}^{\alpha}} \left[ \ln(r) \bigl( \Res_{3}^{(1)} \prod_{j=0}^{N^{\alpha}}  h_0^{(-1)}(\cdot-p_j,\lambda_0) \bigr) \right. \\ \nonumber
 & & \left. -\bigl( \Res_{3}^{(2)} \prod_{j=0}^{N^{\alpha}}  h_0^{(-1)}(\cdot-p_j,\lambda_0) \bigr) \right]
\end{eqnarray}
Taking into account
\[
 \prod_{j=0}^{N^{\alpha}}  h_0^{(-1)}(\cdot-p_j,\lambda_0) = \prod_{j=0}^{N^{\alpha}} \frac{1}{(w-2-p_j)(w-3-p_j)} ,
\]
we get the resolvents
\[
 \Res_{3} \prod_{j=0}^{N^{\alpha}}  h_0^{(-1)}(\cdot-p_j,\lambda_0) = \prod_{j=0}^{N^{\alpha}-1} \frac{1}{p_j(p_j-1)} \quad (\alpha \in S_{p-2}\lfloor_2 \ \mbox{and} \ p \geq 2) ,
\]
\[
 \Res_{3}^{(1)} \prod_{j=0}^{N^{\alpha}}  h_0^{(-1)}(\cdot-p_j,\lambda_0) = \left\{ \begin{array}{cl} -1 & p =1 \\ -\prod_{j=0}^{N^{\alpha}-2} \frac{1}{p_j(p_j-1)} & p \geq 2 \end{array} \right. \quad (\alpha \in S_{p-1}\lfloor_1) ,
\]
\[
 \Res_{3}^{(2)} \prod_{j=0}^{N^{\alpha}}  h_0^{(-1)}(\cdot-p_j,\lambda_0) = \left\{ \begin{array}{cl} 0 & p =1 \\ -\prod_{j=0}^{N^{\alpha}-2} \frac{1}{p_j(p_j-1)} \sum_{j=0}^{N^{\alpha}-2} \left( \frac{1}{p_j-1} + \frac{1}{p_j} \right) & p \geq 2 \end{array} \right. \quad (\alpha \in S_{p-1}\lfloor_1) .
\]
Plugging these expressions into (\ref{kp3a}), we get
\[
 k_0(r) = -\frac{1}{4\pi r} , \quad k_1(r) = \frac{1}{4\pi r} 2Z \ln(r) ,
\]
and for $p \geq 2$
\begin{eqnarray}
\label{kp3b}
 k_p(r) & = & -\frac{1}{4\pi r} \sum_{\alpha \in S_{p-2}\lfloor_2} (-1)^{N_Z^{\alpha}} 2^{N^{\alpha}} Z^{N_Z^{\alpha}} \bigl( \kappa^2 \bigr)^{N_{\kappa}^{\alpha}} \prod_{j=0}^{N^{\alpha}-1} \frac{1}{p_j(p_j-1)} \\ \nonumber
 &  & -\frac{1}{4\pi r} \sum_{\alpha \in S_{p-1}\lfloor_1} (-1)^{N_Z^{\alpha}} 2^{N^{\alpha}} Z^{N_Z^{\alpha}} \bigl( \kappa^2 \bigr)^{N_{\kappa}^{\alpha}} \prod_{j=0}^{N^{\alpha}-2} \frac{1}{p_j(p_j-1)} \left[ \ln(r)
 -\sum_{j=0}^{N^{\alpha}-2} \left( \frac{1}{p_j-1} + \frac{1}{p_j} \right) \right]
\end{eqnarray}

\begin{lemma}
Let the function $k_{\Aop}(Z,\kappa,r)$ be given by the series
\begin{equation}
 k_{\Aop}(Z,\kappa,r) = \sum_{p=0}^{\infty} r^p k_p(r) .
\label{kAop}
\end{equation}
The series converges absolutely, i.e., $k_{\Aop}$ is well defined, for all values of $r \in \mathbb{R}_+$\footnote{Actually, it represents a regular function $k_{\Aop}(Z,\kappa,z)$ in the complex $z$-plane cut along the negative real axis.}. For fixed $r>0$ it is an entire function of the parameters $Z,\kappa \in \mathbb{C}$.
\label{lemmaZkappa}
\end{lemma}
\begin{proof}
Let us first estimate the cardinality of a set $|S_q|$, $q \geq 0$, which is given by
\[
 |S_q| = \sum_{N^{\kappa}=0}^{\lfloor q/2 \rfloor} \begin{pmatrix} q-N^{\kappa} \\ N^{\kappa} \end{pmatrix} = F(q) ,
\] 
where $F(q)$ is the $q$'th Fibonacci number, which for large $q$ have the asymptotic behaviour
\[
 F(q) \sim \frac{a^q}{\sqrt{5}} \quad \mbox{with} \ a=\frac{1+\sqrt{5}}{2} .
\]
Furthermore, we can estimate the products 
\[
 \frac{1}{p!(p-1)!} \leq \prod_{j=0}^{N^{\alpha}-1} \frac{1}{p_j(p_j-1)}, \prod_{j=0}^{N^{\alpha}-2} \frac{1}{p_j(p_j-1)} \leq \frac{1}{p!} ,
\]
and finally the sums
\[
 \sum_{j=0}^{N^{\alpha}-2} \left( \frac{1}{p_j-1} + \frac{1}{p_j} \right) \leq H_p +H_{p-1} \sim \ln(p)+\ln(p-1)+2\gamma
\]
Taking $C:=2a\max\{|Z|,|\kappa|\}$, we can prove the lemma by performing a direct comparisation test with respect to the convergent series
\[
 \frac{1}{4\sqrt{5}\pi} \sum_{p=2}^{\infty} r^{p-1} \left( \frac{C^{p-2}}{(p-2)!}
 +\frac{C^{p-1}}{(p-1)!} \bigl[ |\ln(r)|+\ln(p)+\ln(p-1)+2\gamma \bigr] \right) .
\]
\end{proof}
For (\ref{kp3b}) we can consider two limits with respect to the paramters $Z,\kappa$. For $Z=0$, we just recover (\ref{k2m3d} of the shifted Laplacian, whereas for $\kappa=0$ our calculations yield a fundamental solution of the Hamiltonian. 

\begin{proposition}
A fundamental solution of the differential operator (\ref{centralpotential}) for $\kappa=0$, is given
by (\ref{kp3b}), taking $\kappa=0$, i.e., 
\begin{equation}
 k_{\Aop}(Z,0,r) = -\frac{1}{4\pi r} \left[ 1-2Zr \ln(r) +\sum_{p=2}^{\infty} \frac{(-1)^{p} 2^{p} Z^{p}}{p!(p-1)!} r^p \bigl( \ln(r) -(H_{p}+H_{p-1}-1) \bigr) \right] ,
\label{kAopZ}
\end{equation}
where $H_p := \sum_{k=1}^p \frac{1}{k}$ denotes the $p$-th harmonic number. Furthermore,
$k_{\Hop}(r)=-\frac{1}{2}k_{\Aop}(r)$ is a fundamental solution of the corresponding Hamiltonian.
\label{propZ0}
\end{proposition}
\begin{proof}
The absolute convergence of the series (\ref{kAopZ}) for any $r>0$ follows from Lemma \ref{lemmaZkappa}.   

In order to show $\Aop |_{\kappa=0} \Pf k_{\Aop}=\delta$, we have to consider, cf.~Appendix \ref{appendixpolardistributions},
\begin{eqnarray}
\nonumber
\lefteqn{\int_{\mathbb{R}^3} w(x) \Aop |_{\kappa=0} \Pf k_{\Aop}(Z,0,r) \, dx} \\ \nonumber
 & = & 4\pi \int_{0}^{\infty} \left[ \biggl( 1+r \frac{\partial}{\partial r} \biggr)^2 w_{0}(r) -\biggl( 1+r \frac{\partial}{\partial r} \biggr) w_{0}(r) +2rZ w_{0}(r) \right] k_{\Aop}(Z,0,r) dr \\ \label{PfkA}
 & = & 4\pi \int_{0}^{\infty} \left[ \biggl( r \frac{\partial}{\partial r} \biggr)^2 w_{0}(r) +\biggl( r \frac{\partial}{\partial r} \biggr) w_{0}(r) +2rZ w_{0}(r) \right] k_{\Aop}(Z,0,r) dr ,
\end{eqnarray}
with $w_0(r)=\frac{1}{4\pi} \int_{S^2} w(r,\phi) \mu(\phi) \, d\phi$.
Straightforward calculations yield
\[
 \int_{0}^{\infty} \left( r \frac{\partial}{\partial r} w_0(r) \right) r^{-1} \, dr = -w_0(0)
\]
\[
 \int_{0}^{\infty} \left( r \frac{\partial}{\partial r} \right) \left[ r \frac{\partial}{\partial r} w_0(r) \right] r^{-1} \, dr = 0
\]
and for $q=0,1,2,\ldots$
\[
 \int_{0}^{\infty} \left( r \frac{\partial}{\partial r} w_0(r) \right) r^q \, dr = -(q+1)\int_{0}^{\infty} r^q w_0(r) \, dr
\]
\[
 \int_{0}^{\infty} \left( r \frac{\partial}{\partial r} w_0(r) \right) r^q \ln(r) \, dr = -\int_{0}^{\infty} \bigl( (q+1)\ln(r) +1\bigr) r^q w_0(r) \, dr
\]

\[
 \int_{0}^{\infty} \left( r \frac{\partial}{\partial r} \right) \left[ r \frac{\partial}{\partial r} w_0(r) \right] r^q \, dr = (q+1)^2\int_{0}^{\infty} r^q w_0(r) \, dr
\]
\[
 \int_{0}^{\infty} \left( r \frac{\partial}{\partial r} \right) \left[ r \frac{\partial}{\partial r} w_0(r) \right] r^q \ln(r) \, dr = (q+1)\int_{0}^{\infty} \bigl( (q+1)\ln(r) +2 \bigr) r^q w_0(r) \, dr
\]
Interting these expressions into (\ref{PfkA}) yields
\begin{eqnarray*}
\lefteqn{\int_{\mathbb{R}^n} w(x) \Aop \Pf k_{\Aop}(Z,0,r) \, dx} \\
 & = & w_0(0) +2^2Z^2 \int_{0}^{\infty} r \ln(r) w_0(r) \, dr
 -\sum_{p=2}^{\infty} \frac{(-1)^{p} 2^{p} Z^{p}}{p!(p-1)!} \left[ p(p-1) \int_{0}^{\infty} r^{p-1} \ln(r) w_0(r) \, dr \right. \\
 & & +(2p-1) \int_{0}^{\infty} r^{p-1} w_0(r) \, dr +2Z \int_{0}^{\infty} r^{p} \ln(r) w_0(r) \, dr \\
 & & \left. -\bigl( H_{p}+H_{p-1}-1 \bigr) \left( p(p-1) \int_{0}^{\infty} r^{p-1} w_0(r) \, dr +2Z \int_{0}^{\infty} r^{p} \ln(r) w_0(r) \, dr \right) \right] \\
 & = & w_0(0) +\sum_{p=3}^{\infty} \frac{(-1)^{p} 2^{p} Z^{p}}{p!(p-1)!} \biggl[ -\bigl( H_{p-1}+H_{p-2}-1 \bigr) p(p-1) \\
 & & -(2p-1) +\bigl( H_{p}+H_{p-1}-1 \bigr) p(p-1) \biggr] \int_{0}^{\infty} r^{p-1} w_0(r) \, dr \\
 & = & w_0(0) \\
 & = & w(0)
\end{eqnarray*}
\end{proof}

Now, let us consider the general case  
\begin{proposition}
The regular distribution $\Pf k_{\Aop}(Z,\kappa,r)$, given by (\ref{kp3b}) and (\ref{kAop}) is a
fundamental solution of the differential operator (\ref{centralpotential}) for $Z,\kappa \in \mathbb{C}$.
\end{proposition}
\begin{proof}
Similar to the proof of Proposition \ref{propZ0}, we have to consider integrals
\begin{eqnarray}
\nonumber
\lefteqn{\int_{\mathbb{R}^3} w(x) \Aop \Pf k_{\Aop}(Z,\kappa,r) \, dx} \\ \label{PFkAgen}
 & = & 4\pi \int_{0}^{\infty} \left[ \biggl( r \frac{\partial}{\partial r} \biggr)^2 w_{0}(r) +\biggl( r \frac{\partial}{\partial r} \biggr) w_{0}(r) +2rZ w_{0}(r) -2r^2\kappa^2 w_{0}(r) \right] k_{\Aop}(Z,\kappa,r) dr ,
\end{eqnarray}
with $w \in {\cal D}(\mathbb{R}^3)$. Using the integrals from Proposition \ref{propZ0}, we obtain
\begin{eqnarray*}
\lefteqn{\int_{\mathbb{R}^3} w(x) \Aop \Pf k_{\Aop}(Z,\kappa,r) \, dx} \\
 & = & w_0(0) +2\kappa^2 \int_0^{\infty} r w_0(r) \, dr +2^2Z^2 \int_0^{\infty} r \ln(r) w_0(r) \, dr -2^2\kappa^2Z \int_0^{\infty} r^2 \ln(r) w_0(r) \, dr \\
 & & -\sum_{p=2}^{\infty} \biggl\{ \sum_{\alpha \in S_{p-2}\lfloor_2} (-1)^{N_Z^{\alpha}} 2^{N^{\alpha}} Z^{N_Z^{\alpha}} \bigl( \kappa^2 \bigr)^{N_{\kappa}^{\alpha}} \prod_{j=0}^{N^{\alpha}-1} \frac{1}{p_j(p_j-1)} 
\left( p(p-1) \int_0^{\infty} r^{p-1} w _0(r) \, dr \right. \\
 & & \left. +2Z \int_0^{\infty} r^{p} w _0(r) \, dr -2\kappa^2 \int_0^{\infty} r^{p+1} w _0(r) \, dr \right) \\
 & & +\sum_{\alpha \in S_{p-1}\lfloor_1} (-1)^{N_Z^{\alpha}} 2^{N^{\alpha}} Z^{N_Z^{\alpha}} \bigl( \kappa^2 \bigr)^{N_{\kappa}^{\alpha}} \prod_{j=0}^{N^{\alpha}-2} \frac{1}{p_j(p_j-1)} \left[ p(p-1) \int_0^{\infty} r^{p-1} \ln(r) w _0(r) \, dr \right. \\
 & & \left. + (2p-1) \int_0^{\infty} r^{p-1} w _0(r) \, dr +2Z \int_0^{\infty} r^{p} \ln(r) w _0(r) \, dr -2\kappa^2 \int_0^{\infty} r^{p+1} \ln(r) w _0(r) \, dr \right. \\
 & &  -\left( \sum_{j=0}^{N^{\alpha}-2} \biggl( \frac{1}{p_j-1} + \frac{1}{p_j} \biggr) \right) \left( p(p-1) \int_0^{\infty} r^{p-1} w _0(r) \, dr \right. \\
 & & \left. \left. +2Z \int_0^{\infty} r^{p} w _0(r) \, dr -2\kappa^2 \int_0^{\infty} r^{p+1} w _0(r) \, dr \right) \right] \biggr\}
\end{eqnarray*}
In order to prove the proposition, we can compare the coefficients of the integrals $\int_0^{\infty} r^{p-1} w _0(r) \, dr$ and $\int_0^{\infty} r^{p-1} \ln(r) w _0(r) \, dr$, with $p=1,2,\ldots$. 
Let us start with the integrals $\int_0^{\infty} r^{q} w _0(r) \, dr$. It can be easily seen that the two $p=2$ terms cancel each other. For $p \geq 3$ we get the coefficients
\begin{eqnarray*}
\lefteqn{\int_0^{\infty} r^{p-1} w _0(r) \, dr:} \\
 & & -\sum_{\alpha \in S_{p-2}\lfloor_2} (-1)^{N_Z^{\alpha}} 2^{N^{\alpha}} Z^{N_Z^{\alpha}} \bigl( \kappa^2 \bigr)^{N_{\kappa}^{\alpha}} \prod_{j=0}^{N^{\alpha}-1} \frac{1}{p_j(p_j-1)} p(p-1) \\
 & & -\sum_{\alpha \in S_{p-3}\lfloor_2} (-1)^{N_Z^{\alpha}} 2^{N^{\alpha}} Z^{N_Z^{\alpha}} \bigl( \kappa^2 \bigr)^{N_{\kappa}^{\alpha}} \prod_{j=0}^{N^{\alpha}-1} \frac{1}{(p-1)_j((p-1)_j-1)} 2Z \\
 & & +\sum_{\alpha \in S_{p-4}\lfloor_2} (-1)^{N_Z^{\alpha}} 2^{N^{\alpha}} Z^{N_Z^{\alpha}} \bigl( \kappa^2 \bigr)^{N_{\kappa}^{\alpha}} \prod_{j=0}^{N^{\alpha}-1} \frac{1}{(p-2)_j((p-2)_j-1)} 2\kappa^2 \\
 & & -\sum_{\alpha \in S_{p-1}\lfloor_1} (-1)^{N_Z^{\alpha}} 2^{N^{\alpha}} Z^{N_Z^{\alpha}} \bigl( \kappa^2 \bigr)^{N_{\kappa}^{\alpha}} \prod_{j=0}^{N^{\alpha}-2} \frac{1}{p_j(p_j-1)} \left[ (2p-1)
 -\left( \sum_{j=0}^{N^{\alpha}-2} \biggl( \frac{1}{p_j-1} + \frac{1}{p_j} \biggr) \right) p(p-1) \right] \\
  & & +\sum_{\alpha \in S_{p-2}\lfloor_1} (-1)^{N_Z^{\alpha}} 2^{N^{\alpha}} Z^{N_Z^{\alpha}} \bigl( \kappa^2 \bigr)^{N_{\kappa}^{\alpha}} \prod_{j=0}^{N^{\alpha}-2} \frac{1}{(p-1)_j((p-1)_j-1)}
 \left( \sum_{j=0}^{N^{\alpha}-2} \biggl( \frac{1}{(p-1)_j-1} + \frac{1}{(p-1)_j} \biggr) \right) 2Z \\
 & & -\sum_{\alpha \in S_{p-3}\lfloor_1} (-1)^{N_Z^{\alpha}} 2^{N^{\alpha}} Z^{N_Z^{\alpha}} \bigl( \kappa^2 \bigr)^{N_{\kappa}^{\alpha}} \prod_{j=0}^{N^{\alpha}-2} \frac{1}{(p-2)_j((p-2)_j-1)}
 \left( \sum_{j=0}^{N^{\alpha}-2} \biggl( \frac{1}{(p-2)_j-1} + \frac{1}{(p-2)_j} \biggr) \right) 2\kappa^2 \\
\end{eqnarray*}
In order to prove, that the terms in these coefficients cancel each other, let us consider the first three terms and the last three terms, separately. The second and third term add up to
\begin{eqnarray*}
\lefteqn{\ } \\
 & & -\sum_{\alpha \in S_{p-3}\lfloor_2} (-1)^{N_Z^{\alpha}} 2^{N^{\alpha}} Z^{N_Z^{\alpha}} \bigl( \kappa^2 \bigr)^{N_{\kappa}^{\alpha}} \prod_{j=0}^{N^{\alpha}-1} \frac{1}{(p-1)_j((p-1)_j-1)} 2Z \\
 & & +\sum_{\alpha \in S_{p-4}\lfloor_2} (-1)^{N_Z^{\alpha}} 2^{N^{\alpha}} Z^{N_Z^{\alpha}} \bigl( \kappa^2 \bigr)^{N_{\kappa}^{\alpha}} \prod_{j=0}^{N^{\alpha}-1} \frac{1}{(p-2)_j((p-2)_j-1)} 2\kappa^2 \\
 & = & \sum_{\alpha \in 1 \rfloor S_{p-3}\lfloor_2} (-1)^{N_Z^{\alpha}} 2^{N^{\alpha}} Z^{N_Z^{\alpha}} \bigl( \kappa^2 \bigr)^{N_{\kappa}^{\alpha}} \prod_{j=0}^{N^{\alpha}-1} \frac{1}{p_j(p_j-1)} p(p-1) \\
 & & +\sum_{\alpha \in 2 \rfloor S_{p-4}\lfloor_2} (-1)^{N_Z^{\alpha}} 2^{N^{\alpha}} Z^{N_Z^{\alpha}} \bigl( \kappa^2 \bigr)^{N_{\kappa}^{\alpha}} \prod_{j=0}^{N^{\alpha}-1} \frac{1}{p_j(p_j-1)} p(p-1) \\
 & = & \sum_{\alpha \in S_{p-2}\lfloor_2} (-1)^{N_Z^{\alpha}} 2^{N^{\alpha}} Z^{N_Z^{\alpha}} \bigl( \kappa^2 \bigr)^{N_{\kappa}^{\alpha}} \prod_{j=0}^{N^{\alpha}-1} \frac{1}{p_j(p_j-1)} p(p-1)
\end{eqnarray*}
and therefore cancel the first term. By a similar argument, we get for the last and second last term
\begin{eqnarray*}
\lefteqn{\ } \\
  & & +\sum_{\alpha \in S_{p-2}\lfloor_1} (-1)^{N_Z^{\alpha}} 2^{N^{\alpha}} Z^{N_Z^{\alpha}} \bigl( \kappa^2 \bigr)^{N_{\kappa}^{\alpha}} \prod_{j=0}^{N^{\alpha}-2} \frac{1}{(p-1)_j((p-1)_j-1)}
 \left( \sum_{j=0}^{N^{\alpha}-2} \biggl( \frac{1}{(p-1)_j-1} + \frac{1}{(p-1)_j} \biggr) \right) 2Z \\
 & & -\sum_{\alpha \in S_{p-3}\lfloor_1} (-1)^{N_Z^{\alpha}} 2^{N^{\alpha}} Z^{N_Z^{\alpha}} \bigl( \kappa^2 \bigr)^{N_{\kappa}^{\alpha}} \prod_{j=0}^{N^{\alpha}-2} \frac{1}{(p-2)_j((p-2)_j-1)}
 \left( \sum_{j=0}^{N^{\alpha}-2} \biggl( \frac{1}{(p-2)_j-1} + \frac{1}{(p-2)_j} \biggr) \right) 2\kappa^2 \\
  & = & -\sum_{\alpha \in 1\rfloor S_{p-2}\lfloor_1} (-1)^{N_Z^{\alpha}} 2^{N^{\alpha}} Z^{N_Z^{\alpha}} \bigl( \kappa^2 \bigr)^{N_{\kappa}^{\alpha}} \prod_{j=0}^{N^{\alpha}-2} \frac{1}{p_j(p_j-1)}
 \left( \sum_{j=0}^{N^{\alpha}-2} \biggl( \frac{1}{p_j-1} + \frac{1}{p_j} \biggr) -\frac{1}{p} -\frac{1}{p-1} \right) p(p-1)\\
 & & -\sum_{\alpha \in 2\rfloor S_{p-3}\lfloor_1} (-1)^{N_Z^{\alpha}} 2^{N^{\alpha}} Z^{N_Z^{\alpha}} \bigl( \kappa^2 \bigr)^{N_{\kappa}^{\alpha}} \prod_{j=0}^{N^{\alpha}-2} \frac{1}{p_j(p_j-1)}
 \left( \sum_{j=0}^{N^{\alpha}-2} \biggl( \frac{1}{p_j-1} + \frac{1}{p_j} \biggr) -\frac{1}{p} -\frac{1}{p-1}  \right) p(p-1)\\
 & = & -\sum_{\alpha \in S_{p-1}\lfloor_1} (-1)^{N_Z^{\alpha}} 2^{N^{\alpha}} Z^{N_Z^{\alpha}} \bigl( \kappa^2 \bigr)^{N_{\kappa}^{\alpha}} \prod_{j=0}^{N^{\alpha}-2} \frac{1}{p_j(p_j-1)}
 \left[ \left( \sum_{j=0}^{N^{\alpha}-2} \biggl( \frac{1}{p_j-1} + \frac{1}{p_j} \biggr) \right) p(p-1) -(2p-1) \right]
\end{eqnarray*}
and therefore cancel the third last term. 
It remains to consider the coefficients of the integrals $\int_0^{\infty} r^{p-1} \ln(r) w _0(r) \, dr$, with $p=2,3,\ldots$. It can be seen by inspection that the $p=2$ and $p=3$ terms cancel each other. For $p \geq 4$ we get the coefficients
\begin{eqnarray*}
\lefteqn{\int_0^{\infty} r^{p-1} \ln(r) w _0(r) \, dr:} \\
 & & -\sum_{\alpha \in S_{p-1}\lfloor_1} (-1)^{N_Z^{\alpha}} 2^{N^{\alpha}} Z^{N_Z^{\alpha}} \bigl( \kappa^2 \bigr)^{N_{\kappa}^{\alpha}} \prod_{j=0}^{N^{\alpha}-2} \frac{1}{p_j(p_j-1)} p(p-1) \\
 & &  -\sum_{\alpha \in S_{p-2}\lfloor_1} (-1)^{N_Z^{\alpha}} 2^{N^{\alpha}} Z^{N_Z^{\alpha}} \bigl( \kappa^2 \bigr)^{N_{\kappa}^{\alpha}} \prod_{j=0}^{N^{\alpha}-2} \frac{1}{(p-1)_j((p-1)_j-1)} 2Z \\
 & &  +\sum_{\alpha \in S_{p-3}\lfloor_1} (-1)^{N_Z^{\alpha}} 2^{N^{\alpha}} Z^{N_Z^{\alpha}} \bigl( \kappa^2 \bigr)^{N_{\kappa}^{\alpha}} \prod_{j=0}^{N^{\alpha}-2} \frac{1}{(p-2)_j((p-2)_j-1)} 2\kappa^2 \\ 
\end{eqnarray*}
By the same argument given before, it can be shown that the second and third term cancel the first term, and therefore also these coefficients vanish all together. Summing up, we get
\[
 \int_{\mathbb{R}^n} w(x) \Aop \Pf k_{\Aop}(Z,0,r) \, dx =w_0(0) =w(0) ,
\]
which finishes the proof.
\end{proof}

\vspace{2cm}
\appendix
\noindent {\bf \Large Appendices}
\section{Regularization of fundamental solutions}
\label{app-regfs}
Loosely speaking, any reasonable definition of ellipticity for differential operators in a pseudodifferential calculus implies existence of a parametrix. In the framework of the singular pseudodifferential calculus, considered in the present work, the notion of ellipticity involves a whole hierarchy of symbols associated to a differential operator, cf.~Chapter 10 of \cite{HS08} for a detailed discussion. For unbounded domains, e.g.~$\mathbb{R}^n$, one has to take into account the exit behaviour to infinity, as a result, the Laplacian $\Delta_n$ is not elliptic in $\mathbb{R}^n$, it is only the shifted Laplacian $\Delta_n-\kappa^2$, which satisfies the ellipticity conditions. A possible resolution to this problem is to construct at first regularized parametrices for the shifted differential operators which have the correct exit behaviour at infinity and from these, the corresponding fundamental solutions or Green's functions. Afterwards one considers the limit $\kappa \rightarrow 0$ for the regularization parameter $\kappa$, and sees whether one gets the fundamental solution or Green's function of the original problem. 

Let us illustrate such an approach for the Laplacian $\Delta_n$ in $\mathbb{R}^n$ in the case $n \geq 3$. 
A fundamental solution of the shifted Laplacian $\Delta_n-\kappa^2$, see e.g.~Schwartz \cite{Schwartz}[Section II, \S 3], is given by 
\begin{equation}
 u_{n,\kappa} = \Pf \left[ -(2\pi)^{-\frac{n}{2}} \kappa^{\frac{n-2}{2}} r^{\frac{2-n}{2}} K_{\frac{n-2}{2}}(\kappa r) \right] ,
\label{fndsolLk}
\end{equation}
where $K_{\frac{n-2}{2}}$, $n=3,4\ldots$, denote modified Bessel functions of the second kind. The modified Bessel functions have an asymptotic behaviour for $r \rightarrow 0$, cf.~\cite{AS}, of the form
\[
 K_{\frac{n-2}{2}}(\kappa r) \sim \tfrac{1}{2} \Gamma \bigl( \tfrac{n-2}{2} \bigr) \left( \tfrac{1}{2} \kappa r \right)^{-\frac{n-2}{2}} .
\]
Using the identities
\[ 
\Gamma \bigl( \tfrac{n-2}{2} \bigr) = \tfrac{2}{n-2} \Gamma \bigl( \tfrac{n}{2} \bigr) \ \mbox{and} \ \omega_n=\frac{2\pi^{\frac{n}{2}}}{\Gamma \bigl( \tfrac{n}{2} \bigr)}  \ \mbox{(area of $S^{n-1}$)} ,
\]
we get
\[
 \lim_{\kappa \rightarrow 0} u_{n,\kappa} = -\frac{1}{(n-2)\omega_n}r^{2-n}
\]
which agrees with the fundamental solution of the Laplacian.

For the shifted Laplacian $\Delta_n -\kappa^2$ in even dimension $n$, we want to consider possible variations of the canonical fundamental solution (\ref{fndsolLk}) by subtracting some smooth terms which belong to the kernel of the shifted Laplacian.
For the modified Bessel functions of integer order $K_{\frac{n-2}{2}}$, we consider the series, cf.~\cite{AS}[9.6.11], and obtain
\begin{eqnarray}
\nonumber
 -(2\pi)^{-\frac{n}{2}} \kappa^{\frac{n-2}{2}} r^{\frac{2-n}{2}} K_{\frac{n-2}{2}}(\kappa r) & = & - \frac{1}{4\pi^{\frac{n}{2}}} r^{2-n} \sum_{k=0}^{\frac{n-4}{2}} (-1)^k \frac{(\frac{n-4}{2} -k)!}{4^k k!} (\kappa r)^{2k} \\ \label{canu}
 & & +(-1)^{\frac{n-2}{2}} (2\pi)^{-\frac{n}{2}} \kappa^{\frac{n-2}{2}} r^{\frac{2-n}{2}} I_{\frac{n-2}{2}}(\kappa r) \ln \bigl( \tfrac{1}{2} \kappa r \bigr) \\ \nonumber
 & & +(-1)^{\frac{n}{2}} \kappa^{n-2} (4\pi)^{-\frac{n}{2}} \sum_{k=0}^{\infty} \bigl( \psi(k+1)+\psi(\tfrac{n}{2}+k) \bigr) \frac{1}{4^k k! (\frac{n-2}{2} +k)!} (\kappa r)^{2k}
\end{eqnarray}
with Psi function
\[
 \psi(1) = -\gamma, \quad \psi(m) = -\gamma +\sum_{j=1}^{m-1} j^{-1} , \ m=2,3,\ldots
\]
The function, cf.~\cite{AS}[9.6.10],
\begin{equation}
 r^{\frac{2-n}{2}} I_{\frac{n-2}{2}}(\kappa r) = \left( \tfrac{\kappa}{2} \right)^{\frac{n-2}{2}} \sum_{k=0}^{\infty}  \frac{1}{4^k k! (\frac{n-2}{2} +k)!} (\kappa r)^{2k}
\label{rI}
\end{equation}
can be obviously extended to a smooth function in $\mathbb{R}^n$, which belongs to the kernel of the shifted Laplacian as can be seen by an explicit calculation in polar coordinates
\begin{eqnarray*}
 \bigl( \tilde{\Delta}_{n} -\kappa^2 \bigr) r^{\frac{2-n}{2}} I_{\frac{n-2}{2}}(\kappa r) & = &  \frac{1}{r^2} \biggl[ \biggl(-r \frac{\partial}{\partial r} \biggr)^2 - (n-2)\biggl(-r \frac{\partial}{\partial r} \biggr) + \Delta_{S^{n-1}} \biggr] r^{\frac{2-n}{2}} I_{\frac{n-2}{2}}(\kappa r) \\
  & = & r^{-\frac{n+1}{2}} \left[ (\kappa r)^2 \partial_r^2 I_{\frac{n-2}{2}}(\kappa r) +(\kappa r) \partial_r I_{\frac{n-2}{2}}(\kappa r) -\bigl( \left( \tfrac{n-2}{2} \right)^2 +(\kappa r)^2 \bigr) I_{\frac{n-2}{2}}(\kappa r) \right] \\
  & = & 0 ,
\end{eqnarray*}
which continuously extends to $\mathbb{R}^n$ because of the smoothness of the function. Therefore, we can subtract from (\ref{canu}) the terms
\[
 (-1)^{\frac{n-2}{2}} (2\pi)^{-\frac{n}{2}} \kappa^{\frac{n-2}{2}} r^{\frac{2-n}{2}} I_{\frac{n-2}{2}}(\kappa r) \ln \bigl( \tfrac{1}{2} \kappa \bigr)
\]
and
\[
 -(2\gamma +\alpha) (-1)^{\frac{n}{2}} \kappa^{n-2} (4\pi)^{-\frac{n}{2}} \sum_{k=0}^{\infty} \frac{1}{4^k k! (\frac{n-2}{2} +k)!} (\kappa r)^{2k} \quad \alpha \in \mathbb{R}
\]
in order to obtain the modified fundamental solution $\Pf k_{mod}(r)$ with
\begin{eqnarray}
\nonumber
 k_{mod}(r) & = & - \frac{1}{4\pi^{\frac{n}{2}}} r^{2-n} \sum_{k=0}^{\frac{n-4}{2}} (-1)^k \frac{(\frac{n-4}{2} -k)!}{4^k k!} (\kappa r)^{2k} \\ \label{modu}
 & & +(-1)^{\frac{n-2}{2}} (2\pi)^{-\frac{n}{2}} \kappa^{\frac{n-2}{2}} r^{\frac{2-n}{2}} I_{\frac{n-2}{2}}(\kappa r) \ln(r) \\ \nonumber
 & & +(-1)^{\frac{n}{2}} \kappa^{n-2} (4\pi)^{-\frac{n}{2}} \sum_{k=0}^{\infty} \left( -\alpha +\sum_{j=1}^{k} j^{-1} +\sum_{j=1}^{\frac{n-2}{2}+k} j^{-1} \right) \frac{1}{4^k k! (\frac{n-2}{2} +k)!} (\kappa r)^{2k}
\end{eqnarray}

\section{Derivatives of distributions in polar coordinates}
\label{appendixpolardistributions}
Within the present work, we consider regular distributions $\Pf k$ in ${\cal D}'(\mathbb{R}^n)$, with functions $k: C^{\wedge}(X) \rightarrow \mathbb{R}$ which are ${\cal O}(r^{2-n})$ for $r \rightarrow 0$. In order to show that $\Pf k$ represents a fundamental solution of a partial differential operator $\Aop$, with corresponding operator $\tilde{\Aop}$ in the cone algebra, we have to consider the distributional derivative $\Aop \Pf k$, i.e., the distribution defined by
\[
 \int_{\mathbb{R}^n} g(x) \Aop \Pf k(x) \, dx := \int_{\mathbb{R}^n} \bigl( \Aop^{\ast} g(x) \bigr) \Pf k(x) \, dx ,
\]   
where $\Aop^{\ast}$ denotes the formally adjoint operator and $g$ any test function in ${\cal D}(\mathbb{R}^n)$. It should be emphasized, that this definition refers exclusively to cartesian coordinates with corresponding Lebesgue measure $dx$, cf.~Schwartz \cite{Schwartz}[Section II, \S 3]. 

Within the present work, we consider partial differential operators $\Aop$ which correpond to
partial differential operators on the streched cone of the form
\[
 \tilde{\Aop} = r^{-2} \left[ \biggl(-r \frac{\partial}{\partial r} \biggr)^{2} +a_1^{(0)} \biggl(-r \frac{\partial}{\partial r} \biggr) +a_0^{(0)} +b^{(0)} \Lambda_X \right] ,
\]
with coefficients satisfying $a^{(0)}_0 =-a^{(0)}_1(n-2)-(n-2)^2$.
Using $\bigl(-r \frac{\partial}{\partial r} \bigr) =-\bigl( \sum_{i=1}^n x_i \frac{\partial}{\partial x_i} \bigr)$ and the essential self-adjointness of $\Lambda_X$, we get
\begin{eqnarray*}
 \lefteqn{\int_{\mathbb{R}^n} g(x) \Aop \Pf k(x) \, dx} \\
 & = & \int_{\mathbb{R}^n} \left( \left[ \bigl( \frac{\partial}{\partial x_i} x_i \bigr) \bigl( \frac{\partial}{\partial x_j} x_j \bigr) -a_1^{(0)} \bigl( \frac{\partial}{\partial x_i} x_i \bigr) +a_0^{(0)} +b^{(0)} \Lambda_X \right] \frac{g(x)}{|x|^2} \right) \Pf k(x) \, dx \\
 & = & \int_{X} \int_0^{\infty} \left( \left[ \bigl( n+r \frac{\partial}{\partial r} \bigr)^2  -a_1^{(0)} \bigl( n+r \frac{\partial}{\partial r} \bigr) +a_0^{(0)} +b^{(0)} \Lambda_X \right] \frac{\tilde{g}(r,\phi)}{r^2} \right)  k(r,\phi) \, r^{n-1} dr \mu(\phi) d\phi \\
 & = & \int_{X} \int_0^{\infty} \frac{1}{r^2} \biggl( \biggl[ \bigl( (n-2)+r \frac{\partial}{\partial r} \bigr)^2 +a^{(0)}_1 \bigl( (n-2)+r \frac{\partial}{\partial r} \bigr) \\
 & & +a^{(0)}_0 +b^{(0)} \Lambda_X \biggr] \tilde{g}(r,\phi) \biggr) k(r,\phi) \, r^{n-1} dr \mu(\phi) d\phi \\
\end{eqnarray*}
with $\tilde{g} :=\varphi^{\ast}g(r,\phi)$ for $g \in {\cal D}(\mathbb{R}^n)$, cf.~Section \ref{Section-intro}. These integrals exist, because the integrand remains bounded due to our constraint on the coefficients and the ${\cal O}(r^{2-n})$ behaviour of $k$.

\section{Calculation of the parametrix for the shifted Laplace operator in three dimensions}
\label{appendixKkappa}
In this appendix we want to give some technical details concerning the calculation of the kernel of the parametrix (\ref{Kkappa}) for the shifted Laplace operator in three dimensions. According to Proposition \ref{proposition3} all poles of symbols of the parametrix are simple, and formula (\ref{formulaKp}) simplifies to
\begin{eqnarray*}
 K_{2m}(r,\phi|\tilde{r},\tilde{\phi}) & = & -H(r-\tilde{r}) \sum_{\ell=0}^{\infty} \sum_{k=1}^{j_{\ell}} r^{2-w_k(\lambda_{\ell})} \tilde{r}^{w_k(\lambda_{\ell})-n} \bigl( \Res_{w_k(\lambda_{\ell})} h^{(-1)}_{2m}(\cdot,\lambda_{\ell}) \bigr) p_{\ell}(\phi|\tilde{\phi}) \\
  & & +H(\tilde{r} -r) \sum_{\ell=0}^{\infty} \sum_{k=j_{\ell}+1}^{4m+2} r^{2-w_k(\lambda_{\ell})} \tilde{r}^{w_k(\lambda_{\ell})-n} \bigl( \Res_{w_k(\lambda_{\ell})} h^{(-1)}_{2m}(\cdot,\lambda_{\ell}) \bigr) p_{\ell}(\phi|\tilde{\phi}) .
\end{eqnarray*} 
The integration contour $\Gamma_{\frac{7}2-\gamma}$ is taken with $\frac{1}{2} < \gamma < \frac{3}{2}$,
where the poles
\[
 w_{1,j}(\lambda_{\ell} )= 3+\ell+p_j, \ \mbox{with} \ j=0,1,\ldots,m \quad \mbox{and} \quad
 w_{2,j}(\lambda_{\ell}) = 2-\ell+p_j, \ \mbox{with} \ j=0,1,\ldots,m-\lfloor \tfrac{\ell}{2} \rfloor -1
\]
are located right of the integration contour, and the poles
\[
 w_{2,j}(\lambda_{\ell}) = 2-\ell+p_j, \ \mbox{with} \ j=m-\lfloor \tfrac{\ell}{2} \rfloor, \ldots,m ,
\]
are located left of it. Here and in following, we use the notation $p_j := 2m-2j$.

 The symbols (\ref{h2m}) of the parametrix are given by
\[
 h^{(-1)}_{2m}(w,\lambda_{\ell}) = \kappa^{2m} \prod_{j=0}^m \frac{1}{(w-2+\ell-p_j)(w-3-\ell-p_j)} \quad \mbox{for} \ m=0,1,\ldots ,
\]
and we have to calculate the residues at their poles. Let us first consider the case $r > \tilde{r}$, where the poles right to the integration contour contribute.
For these poles, we get
\[
 \Res_{w_j(\lambda_{\ell})} h^{(-1)}_{2m}(\cdot,\lambda_{\ell}) =  \prod_{\substack{k=0\\k \neq j}}^m \frac{1}{p_j-p_k}
 \prod_{k=0}^m \frac{1}{p_j-p_k+1+2\ell}  \quad \mbox{for} \ j=0,1,\ldots m \quad w_j(\lambda_{\ell}) = 3+\ell+p_j ,
\] 
and
\[
 \Res_{w_j(\lambda_{\ell})} h^{(-1)}_{2m}(\cdot,\lambda_{\ell}) =  \prod_{\substack{k=0\\k \neq j}}^m \frac{1}{p_j-p_k}
 \prod_{k=0}^m \frac{1}{p_j-p_k-1-2\ell}  \quad \mbox{for} \ j=0,1,\ldots m-\lfloor \tfrac{\ell}{2} \rfloor-1 \quad w_j(\lambda_{\ell}) = 2-\ell+p_j ,
\]
For the products, we get
\[
 \prod_{\substack{k=0\\k \neq j}}^m \frac{1}{p_j-p_k} =\frac{(-1)^j}{2^mj!(m-j)!} ,
\] 
\[
 \prod_{k=0}^m \frac{1}{p_j-p_k+1+2\ell} = \left\{ \begin{array}{ll} \frac{(-1)^{\ell-j}}{\bigl(2(j-\ell)-1 \bigr)!! \bigl( (2(m-j+\ell)+1 \bigr)!!} & \mbox{for} \ 2j>2\ell+1 \\
 \frac{\bigr(2(\ell-j)-1 \bigr)!!}{\bigl( (2(m-j+\ell)+1 \bigr)!!} & \mbox{for} \ 2j<2\ell+1 \end{array} \right. ,
\]
\[
 \prod_{k=0}^m \frac{1}{p_j-p_k-1-2\ell} = \left\{ \begin{array}{ll} \frac{(-1)^{\ell+j+1}}{\bigl(2(j+\ell)+1 \bigr)!! \bigl( (2(m-j-\ell)-1 \bigr)!!} & \mbox{for} \ j\leq m-\ell \\
 \frac{(-1)^{m+1}\bigl( (2(j+\ell-m)-1 \bigr)!!}{\bigr(2(j+\ell)+1 \bigr)!!} & \mbox{for} \ j\geq m-\ell+1 \end{array} \right. .
\] 
We can now sum up the kernel of the parametrix, 
\begin{eqnarray}
\nonumber
 \lefteqn{K(r,\phi|\tilde{r},\tilde{\phi})} \\ \label{sum1}
 & = & -\sum_{\ell=0}^{\infty}\sum_{m=0}^{\infty} \sum_{j=0}^{m} r^{-1-\ell+2j} \tilde{r}^{2(m-j)+\ell} \kappa^{2m} \frac{(-1)^j f_j}{2^m j! (m-j)! \bigl( (2(m-j+\ell)+1 \bigr)!!} p_{\ell}(\phi|\tilde{\phi}) \\ \nonumber
 & & -\sum_{\ell=0}^{\infty}\sum_{m=\lfloor \tfrac{\ell}{2} \rfloor+1}^{\infty} \sum_{j=0}^{m-\lfloor \tfrac{\ell}{2} \rfloor-1} r^{\ell+2j} \tilde{r}^{-1-\ell+2(m-j)} \kappa^{2m} \frac{g_{mj}}{2^m j! (m-j)! \bigl( (2(j+\ell)+1 \bigr)!!} p_{\ell}(\phi|\tilde{\phi})
\end{eqnarray}
with
\[
 f_j = \left\{ \begin{array}{ll} \bigl(2(\ell-j)-1 \bigr)!! & \mbox{for} \ 0 \leq j \leq \ell \\
 \frac{(-1)^{j-\ell}}{\bigl( (2(j-\ell)-1 \bigr)!!} & \mbox{for} \ \ell+1 \leq j \end{array} \right. ,
\]
and
\[
 g_{mj} = \left\{ \begin{array}{ll} \frac{(-1)^{\ell+1}}{\bigl(2(m-j-\ell)-1 \bigr)!!} & \mbox{for} \ 0 \leq j \leq m-\ell \\
 (-1)^{j+m+1} \bigl( (2(j+\ell-m)-1 \bigr)!! & \mbox{for} \ m-\ell+1 \leq j \leq m-\lfloor \tfrac{\ell}{2} \rfloor -1 \end{array} \right. ,
\]
By changing the order of summation,  
\[
 \sum_{m=0}^{\infty} \sum_{j=0}^{m} \ \rightarrow \ \sum_{j=0}^{\infty} \sum_{m=j}^{\infty} \xrightarrow{\tilde{m}=m-j} \ \sum_{j=0}^{\infty} \sum_{\tilde{m}=0}^{\infty}
\]
we can achieve a separation of the variables, the first sum in (\ref{sum1}) becomes
\begin{eqnarray}
\nonumber
 \lefteqn{-\sum_{\ell=0}^{\infty}\sum_{j=0}^{\infty} \sum_{\tilde{m}=0}^{\infty}  r^{-1-\ell+2j} \tilde{r}^{2\tilde{m}+\ell} \kappa^{2(\tilde{m}+j)} \frac{(-1)^j f_j}{2^{\tilde{m}+j} j! \tilde{m}! \bigl( (2(\tilde{m}+\ell)+1 \bigr)!!} p_{\ell}(\phi|\tilde{\phi})} \\ \nonumber
 & = & -\sum_{\ell=0}^{\infty}\left( \sum_{j=0}^{\ell} r^{-1-\ell+2j} \kappa^{2j} \frac{(-1)^j   \bigl( (2(j-\ell)-1 \bigr)!!}{2^j j!} +\sum_{j=\ell+1}^{\infty} r^{-1-\ell+2j} \kappa^{2j} \frac{(-1)^{\ell}}{2^j j! \bigl( (2(j-\ell)-1 \bigr)!!} \right) \\ \nonumber
 & & \times \left( \sum_{\tilde{m}=0}^{\infty} \tilde{r}^{2\tilde{m}+\ell} \kappa^{2\tilde{m}} \frac{1}{2^{\tilde{m}} \tilde{m}! \bigl( (2(\tilde{m}+\ell)+1 \bigr)!!} \right) p_{\ell}(\phi|\tilde{\phi}) \\ \label{sum1a}
 & = & -\frac{\pi}{2} \sum_{\ell=0}^{\infty}(-1)^{\ell} \frac{I_{-\ell-\frac{1}{2}}(\kappa r)}{\sqrt{r}} \frac{I_{\ell+\frac{1}{2}}(\kappa \tilde{r})}{\sqrt{\tilde{r}}} p_{\ell}(\phi|\tilde{\phi}) ,
\end{eqnarray}
in the last line, we have identified the power series with modified Bessel function of first kind $I_{\pm\ell\pm\frac{1}{2}}$, cf.~\cite{DLMF}[Eqs.~10.53.4, 10.47.8]. The second sum in (\ref{sum1}) can be rewritten  as
\begin{eqnarray}
\nonumber
 \lefteqn{-\sum_{\ell=0}^{\infty} \sum_{j=0}^{\infty} \sum_{\tilde{m}=\lfloor \tfrac{\ell}{2} \rfloor+1}^{\infty} r^{\ell+2j} \tilde{r}^{-1-\ell+2\tilde{m}} \kappa^{2(\tilde{m}+j)} \frac{g_{\tilde{m}j}}{2^{\tilde{m}+j} j! \tilde{m}! \bigl( (2(j+\ell)+1 \bigr)!!} p_{\ell}(\phi|\tilde{\phi})} \\ \nonumber
 & = & -\sum_{\ell=0}^{\infty} \sum_{j=0}^{\infty} \sum_{\tilde{m}=0}^{\infty} r^{\ell+2j} \tilde{r}^{-1-\ell+2\tilde{m}} \kappa^{2(\tilde{m}+j)} \frac{g_{\tilde{m}j}}{2^{\tilde{m}+j} j! \tilde{m}! \bigl( (2(j+\ell)+1 \bigr)!!} p_{\ell}(\phi|\tilde{\phi}) \\ \nonumber
 & & +\sum_{\ell=0}^{\infty} \sum_{j=0}^{\infty} \sum_{\tilde{m}=0}^{\lfloor \tfrac{\ell}{2} \rfloor} r^{\ell+2j} \tilde{r}^{-1-\ell+2\tilde{m}} \kappa^{2(\tilde{m}+j)} \frac{g_{\tilde{m}j}}{2^{\tilde{m}+j} j! \tilde{m}! \bigl( (2(j+\ell)+1 \bigr)!!} p_{\ell}(\phi|\tilde{\phi}) \\ \nonumber
 & = & \sum_{\ell=0}^{\infty} \left( \sum_{j=0}^{\infty} r^{\ell+2j} \kappa^{2j} \frac{1}{2^{j} j! \bigl( (2(j+\ell)+1 \bigr)!!} \right) \\ \nonumber
 & & \times \left( \sum_{\tilde{m}=0}^{\ell-1} \tilde{r}^{-1-\ell+2\tilde{m}} \kappa^{2\tilde{m}} \frac{(-1)^{\tilde{m}} \bigl( (2(\ell-\tilde{m})-1 \bigr)!!}{2^{\tilde{m}} \tilde{m}!} +\sum_{\tilde{m}=\ell}^{\infty} \tilde{r}^{-1-\ell+2\tilde{m}} \kappa^{2\tilde{m}} \frac{(-1)^{\ell}}{2^{\tilde{m}} \tilde{m}! \bigl( (2(\tilde{m}-\ell)-1 \bigr)!!}\right) p_{\ell}(\phi|\tilde{\phi}) \\ \nonumber
 & & -\sum_{\ell=0}^{\infty} \left( \sum_{j=0}^{\infty} r^{\ell+2j} \kappa^{2j} \frac{1}{2^{j} j! \bigl( (2(j+\ell)+1 \bigr)!!} \right) \left( \sum_{\tilde{m}=0}^{\lfloor \tfrac{\ell}{2} \rfloor} \tilde{r}^{-1-\ell+2\tilde{m}} \kappa^{2\tilde{m}} \frac{(-1)^{\tilde{m}} \bigl( (2(\ell-\tilde{m})-1 \bigr)!!}{2^{\tilde{m}} \tilde{m}!} \right) p_{\ell}(\phi|\tilde{\phi}) \\ \label{sum2}
 & = & \sum_{\ell=0}^{\infty} \left[(-1)^{\ell} \frac{\pi}{2}  \frac{I_{\ell+\frac{1}{2}}(\kappa r)}{\sqrt{r}} \frac{I_{-\ell-\frac{1}{2}}(\kappa \tilde{r})}{\sqrt{\tilde{r}}} -\sqrt{\frac{\pi\kappa}{2}} \frac{I_{\ell+\frac{1}{2}}(\kappa r)}{\sqrt{r}} S_{\ell}(\kappa \tilde{r}) \right] p_{\ell}(\phi|\tilde{\phi}) ,
\end{eqnarray}
with
\[
 \tilde{g}_{\tilde{m}j} = \left\{ \begin{array}{ll} \frac{(-1)^{\ell+1}}{\bigl(2(\tilde{m}-\ell)-1 \bigr)!!} & \mbox{for} \ \ell \leq \tilde{m} \\
 (-1)^{\tilde{m}+1} \bigl( (2(\ell-\tilde{m})-1 \bigr)!! & \mbox{for} \ \tilde{m} \leq \ell-1 \end{array} \right. .
\]
Like before, we identified the power series with modified Bessel function of first kind, except of the finit sum given by
\[
 S_{\ell}(\kappa \tilde{r}) := \sum_{m=0}^{\lfloor \ell/2 \rfloor} (\kappa \tilde{r})^{-1-\ell+2m} \frac{(-1)^m\bigl( 2(\ell-m)-1 \bigr)!!}{2^m m!!} .
\]
Finally, let us consider the case $r < \tilde{r}$, where the poles left to the integration contour contribute, we get
\begin{eqnarray}
\nonumber
 \lefteqn{K(r,\phi|\tilde{r},\tilde{\phi})} \\ \nonumber
 & = & \sum_{\ell=0}^{\infty}\sum_{m=0}^{\lfloor \tfrac{\ell}{2} \rfloor} \sum_{j=0}^{m} r^{\ell+2j} \tilde{r}^{-1-\ell+2(m-j)} \kappa^{2m} \frac{ g_{mj}}{2^m j! (m-j)! \bigl( (2(j+\ell)+1 \bigr)!!} p_{\ell}(\phi|\tilde{\phi}) \\ \nonumber
 & & \sum_{\ell=0}^{\infty}\sum_{m=\lfloor \tfrac{\ell}{2} \rfloor+1}^{\infty} \sum_{j=m-\lfloor \tfrac{\ell}{2} \rfloor}^{m} r^{\ell+2j} \tilde{r}^{-1-\ell+2(m-j)} \kappa^{2m} \frac{g_{mj}}{2^m j! (m-j)! \bigl( (2(j+\ell)+1 \bigr)!!} p_{\ell}(\phi|\tilde{\phi}) \\ \nonumber
 & = & \sum_{\ell=0}^{\infty} \sum_{j=0}^{\infty} \sum_{m=j}^{j+\lfloor \tfrac{\ell}{2} \rfloor}  r^{\ell+2j} \tilde{r}^{-1-\ell+2(m-j)} \kappa^{2m} \frac{g_{mj}}{2^m j! (m-j)! \bigl( (2(j+\ell)+1 \bigr)!!} p_{\ell}(\phi|\tilde{\phi}) \\ \nonumber
 & = & \sum_{\ell=0}^{\infty} \sum_{j=0}^{\infty} \sum_{m=j}^{j+\lfloor \tfrac{\ell}{2} \rfloor}  r^{\ell+2j} \tilde{r}^{-1-\ell+2(m-j)} \kappa^{2m} \frac{(-1)^{j+m+1} \bigl( (2(j+\ell-m)-1 \bigr)!!}{2^m j! (m-j)! \bigl( (2(j+\ell)+1 \bigr)!!} p_{\ell}(\phi|\tilde{\phi}) \\ \nonumber
 & = & \sum_{\ell=0}^{\infty} \sum_{j=0}^{\infty} \sum_{\tilde{m}=0}^{\lfloor \tfrac{\ell}{2} \rfloor}  r^{\ell+2j} \tilde{r}^{-1-\ell+2\tilde{m}} \kappa^{2(\tilde{m}+j)} \frac{(-1)^{\tilde{m}+1} \bigl( (2(\ell-\tilde{m})-1 \bigr)!!}{2^{\tilde{m}+j} j! \tilde{m}! \bigl( (2(j+\ell)+1 \bigr)!!} p_{\ell}(\phi|\tilde{\phi}) \\ \nonumber
 & = & -\sum_{\ell=0}^{\infty} \left( \sum_{j=0}^{\infty} r^{\ell+2j} \kappa^{2j} \frac{1}{2^{j} j! \bigl( (2(j+\ell)+1 \bigr)!!} \right) \left( \sum_{\tilde{m}=0}^{\lfloor \tfrac{\ell}{2} \rfloor} \tilde{r}^{-1-\ell+2\tilde{m}} \kappa^{2\tilde{m}} \frac{(-1)^{\tilde{m}} \bigl( (2(\ell-\tilde{m})-1 \bigr)!!}{2^{\tilde{m}} \tilde{m}!} \right) p_{\ell}(\phi|\tilde{\phi}) \\ \label{sum3}
 & = & -\sum_{\ell=0}^{\infty} \sqrt{\frac{\pi\kappa}{2}} \frac{I_{\ell+\frac{1}{2}}(\kappa r)}{\sqrt{r}} S_{\ell}(\kappa \tilde{r}) p_{\ell}(\phi|\tilde{\phi}) .
\end{eqnarray}   
Summing up the terms (\ref{sum1a}), (\ref{sum2}) and (\ref{sum3}), we get the kernel of the parametrix (\ref{Kkappa}).


\clearpage
\begin{thebibliography}{999}
\bibitem{AS} M.~Abramowitz and I.~A.~Stegun, {\em Handbook of Mathematical Functions},
             (National Bureau of Standards, Applied Mathematics Series - 55, 1972).
\bibitem{ES97} Y.V. Egorov, B.-W. Schulze, {\em Pseudo-Differential Operators, Singularities, Applications}. Birkh\"auser: Basel; 1997.
\bibitem{DLMF} NIST Digital Library of Mathematical Functions. https://dlmf.nist.gov/, Release 1.1.11 of 2023-09-15. F. W. J. Olver, A. B. Olde Daalhuis, D. W. Lozier, B. I. Schneider, R. F. Boisvert, C. W. Clark, B. R. Miller, B. V. Saunders, H. S. Cohl, and M. A. McClain, eds.
\bibitem{EZ18} B. Engquist anf H. Zhao. {\em Approximate separability of the Green's function of the Helmholtz equation in the high-frequency limit}, Commun. Pure Appl. Math. {\bf 61} (2018) 2220-2274.
\bibitem{FHSS11} H.-J. Flad, G. Harutyunyan, R. Schneider and B.-W. Schulze, {\em Explicit Green operators for quantum mechanical Hamiltonians.~I.~The hydrogen atom}, manuscripta math., Vol. {\bf 135} (2011) 497-519.     
\bibitem{FHS16} H.-J. Flad, G. Harutyunyan and B.-W. Schulze, {\em Asymptotic parametrices of elliptic operators}, J. Pseudo-Differ. Oper. Appl. {\bf 7} (2016) 321-363.           
\bibitem{FF22} F. Franceschini and F. Glaudo, {\em Expansion of the fundamental solution of a second order elliptic operator with analytic coefficients}, arXiv:2110.15104v2 [math.AP] (2022) 25 pp.
\bibitem{HS08} Harutyunyan G, Schulze B-W, {\em Elliptic Mixed, Transmission and Singular Crack Problems}. EMS Tracts in Mathematics Vol. {\bf 4}, European Math. Soc: Z\"urich; 2008.
\bibitem{H1} L. H\"ormander, {\em The Analysis of Linear Partial Differential Operators I, 2nd Ed.} (Springer, Berlin, 1990).
\bibitem{H2} L. H\"ormander, {\em The Analysis of Linear Partial Differential Operators II} (Springer, Berlin, 1983)
\bibitem{H3} L. H\"ormander, {\em The Analysis of Linear Partial Differential Operators III} (Springer, Berlin, 1994).
\bibitem{J50} F. John {\em The fundamental solution of linear elliptic differential equations with analytic coefficients}, Comm. Pure Appl. Math. {\bf 3} (1950) 273-304.
\bibitem{K49} K. Kodaira, {\em Harmonic fields in Riemannian manifolds (generalized potential theory)}, Ann. of Math. {\bf 50} (1949) 587-665.
\bibitem{MM65} N.F. Mott and H.S. Massey, {\em The theory of atomic collisions, 3. Edition} (Oxford University Press, 1965).
\bibitem{RS3} M. Reed and B. Simon, {\em Methods of Modern Mathematical Physics, III Scattering Theory} (Academic Press, San Diego, 1979).
\bibitem{S64} R.A. Sack, {\em Generalization of Laplace's expansion to
arbitrary powers and functions of the
distance between two points} J. Math. Phys. {\bf 5}, 245-251 (1964).
\bibitem{STV18} H. Schlichtkrull, P. Trapa and D.A. Vogan Jr., {\em Laplacians on spheres}, arXiv:1803.01267v2 [math.RT] (2018) 49 pp.
\bibitem{Schulze98} Schulze B-W, {\em Boundary Value Problems and Singular Pseudo-Differential Operators}. Wiley: New York; 1998.
\bibitem{Schwartz} L. Schwartz, {\em Th\'{e}orie des Distributions}, (Hermann. Paris, 1978).
\bibitem{Watson} G.N. Watson, {\em A Treatise on the Theory of Bessel Functions,  2nd Ed.}, (Cambridge University Press, 1941).
\end{thebibliography}
\end{document}